\documentclass[11pt,twoside,a4paper]{article}
     
\usepackage{amssymb}
\usepackage{amsmath}
\usepackage{amsthm, color}

\allowdisplaybreaks

\usepackage{amsfonts}
\usepackage{bbm}
\usepackage{verbatim}
\usepackage{url}
\usepackage[T1]{fontenc}

\usepackage{hyperref}
\hypersetup{
    colorlinks=true,
    linkcolor=blue,
    citecolor=blue,
    filecolor=magenta,      
    urlcolor=cyan,
}

\usepackage[
backend=biber,
style=alphabetic,
maxbibnames=10,
sorting=nyt
]{biblatex}
\DeclareUnicodeCharacter{0301}{\"u}

\def\rr{{\mathbb R}}

\newcommand{\R}{\mathbb{R}}
\newcommand{\M}{\mathcal{M}}
\def\cm{{\mathcal M}}

\def\fz{\infty}

\def\lz{\lambda}
\def\ls{\lesssim}
\def\gs{\gtrsim}

\def\tbz{{\triangle_\lz}}
\def\dmz{{dm_\lz}}

\def\lpz{{L^p(\rr_+,\, dm_\lz)}}

\def\r{\right}
\def\lf{\left}

\newtheorem{thm}{Theorem}[section]
\newtheorem{lem}[thm]{Lemma}
\newtheorem{prop}[thm]{Proposition}
\newtheorem{rem}[thm]{Remark}
\newtheorem{cor}[thm]{Corollary}
\newtheorem{defn}[thm]{Definition}

\numberwithin{equation}{section}

\begin{document}

\arraycolsep=1pt

\title{\Large\bf  Muckenhoupt-type weights and the intrinsic structure \\ in Bessel Setting}
\author{Ji Li, Chong-Wei Liang, Fred Yu-Hsiang Lin, and Chun-Yen Shen}

\date{}
\maketitle

\begin{center}
\begin{minipage}{13.5cm}\small

{\noindent  {\bf Abstract:}\  Fix $\lambda>-1/2$ and $\lambda \not=0$. Consider  
the Bessel operator (introduced by Muckenhoupt--Stein)
$\triangle_\lambda:=-\frac{d^2}{dx^2}-\frac{2\lambda}{x}
\frac d{dx}$ on $\mathbb{R_+}:=(0,\infty)$ with $dm_\lambda(x):=x^{2\lambda}dx$ and $dx$ the Lebesgue measure on $\R_+$.  In this paper, we study the Muckenhoupt-type weights which reveal the intrinsic structure in this Bessel setting along the line of Muckenhoupt--Stein and Andersen--Kerman. Besides, exploiting more properties of the weights $A_{p,\lambda}$ introduced by Andersen--Kerman, we introduce a new class $\widetilde{A}_{p,\lambda}$ such that the Hardy--Littlewood maximal function is bounded on the weighted $L^p_w$ space if and only if $w$ is in $\widetilde A_{p,\lambda}$. Moreover, along the line of Coifman--Rochberg--Weiss, we investigate the commutator $[b,R_\lambda]$ with $R_\lambda:=\frac{d}{dx}(\triangle_\lambda)^{-\frac{1}{2}}$ to be the Bessel Riesz transform. We show that for $w\in A_{p,\lambda}$, the commutator $[b, R_\lambda]$ is bounded on weighted $L^p_w$ if and only if $b$ is in the BMO space associated with $\triangle_\lambda$.

}

\end{minipage}
\end{center}

%
%

\section{Introduction and Statement of Main Results}\label{s1}

\subsection{Background}
In 1965, Muckenhoupt and Stein in \cite{muckenhoupt1965classical} introduced the harmonic function theory associated with the Bessel operator $\tbz$, defined by,
\begin{equation*}
\tbz :=-\frac{d^2}{dx^2}-\frac{2\lz}{x}\frac{d}{dx},\quad \lz>0.
\end{equation*}
They developed a theory in the setting of
$\tbz$ which parallels the classical one associated with the usual Laplacian $\triangle$.  Results on the $\lpz$-boundedness of conjugate
functions and fractional integrals associated with $\tbz$ were
obtained, where $p\in[1, \fz)$, $\rr_+:=(0, \fz)$ and $\dmz(x):= x^{2\lz}\,dx$.  Moreover, they also studied the particular Calder\'on--Zygmund operator associated with $\tbz$: the Bessel Riesz transform $R_\lambda:=\frac{d}{dx}(\triangle_\lambda)^{-\frac{1}{2}}$.
Since then, the Bessel context has been extensively studied, including the theory of Hardy and BMO spaces, Bessel Riesz transforms and the commutator estimates;
see for example, \cite{andersen1981weighted}, \cite{betancor2010maximal}, \cite{betancor2007riesz}, \cite{betancor2009littlewood}, \cite{betancor2010mapping}, \cite{kerman1978boundedness}, \cite{villani2008riesz}, \cite{yang2011real}, \cite{duong2017factorization}, \cite{duong2018compactness}, \cite{betancor2015umd}, \cite{betancor2020bellman}, \cite{chen2022square} and the references therein.

It is worth pointing out that in many subsequent works, the triple $(\R_+,|\cdot|, \dmz)$ is treated as a space of homogeneous type in the sense of Coifman and Weiss \cite{coifman1977extensions}. Together with the particular properties associated with the  Bessel operator $\tbz$ and new techniques,  classical results on $\R^d$ have their extensions in the Bessel setting. 

However, in some particular cases, this point of view may not reflect the intrinsic structure of the Bessel operator $\tbz$. For example, when we treat $(\R_+,|\cdot|, \dmz)$ as a space of homogeneous type, the Muckenhoupt $A_p$ weight is defined naturally following the classical setting, and the weighted $L^p$ norm is 
\begin{align}\label{classical Lpw}
 \|f\|_{L^p_w(\R_+)}:=\bigg( \int_{\R_+} |f(x)|^p w(x) \dmz(x)\bigg)^{\frac{1}{p}},\quad 1<p<\infty,\quad w\in A_p.  
\end{align}

This is general enough to cover many more settings but does not fully capture the specific intrinsic structure of the Bessel Riesz transform. 

In $1981$, Andersen and Kerman \cite{andersen1981weighted}  introduced a new Muckenhoupt-type weights $A_{p,\lambda}, \ 1<p<\infty,$ associated with the Bessel operator $\tbz$, defined as follows:
 
 $w\in A_{p,\lambda}$ if there exists a constant $C > 0$, such that for all $0<a<b<\infty$, we have
\begin{align}\label{Ap lambda}
\bigg(\int_{a}^{b}t^{p}w(t)dt\bigg)\bigg(\int_{a}^{b}t^{2\lambda p'}w(t)^{-\frac{1}{p-1}}dt\bigg)^{p-1}<C(b^{2(\lambda+1)}-a^{2(\lambda +1)})^{p}.
\end{align}
\noindent For  $w\in A_{p,\lambda}$, the $L^p(\R_+,w)$ norm is defined as follows
\begin{align}\label{new Lpw}
\|f\|_{L^p(\R_+,w)}:=\bigg( \int_{\R_+} |f(x)|^p w(x) dx\bigg)^{\frac{1}{p}}.
\end{align}

Then they revealed the intrinsic structure of the Bessel Riesz transform $R_\lambda$ by proving that: for $\lambda>-{\frac{1}{2}},\neq 0 $, $1<p<\infty$, and a nonnegative measurable function $w$,
$R_\lambda$ is bounded on $L^p(\R_+,w)$ {\bf if and only if } $w\in A_{p,\lambda}$.

\subsection{Aims and Questions}

The aim of this paper is to continue the line of Muckenhoupt--Stein  \cite{muckenhoupt1965classical} and Andersen--Kerman \cite{andersen1981weighted} to study the Muckenhoupt-type weights in the Bessel Setting. Analogous to the classical Muckenhoupt $A_p$ theory, the characterization of Muckenhoupt-type weights via suitable maximal function is the first main concern.  Thus, the first natural question that we focus on is as follows.

\medskip
\noindent{\bf Question 1:} What is the suitable Hardy--Littlewood maximal function to characterize the Muckenhoupt-type weights in the Bessel Setting?
\medskip

Next, we would like to mention that in a recent result by Duong, Wick, Yang ,and the first author \cite{duong2017factorization}, the $\lpz$-boundedness of the commutator
$[b, R_\lambda]$ was fully characterized via using the standard BMO space on the triple $(\R_+,|\cdot|, \dmz)$. This generalized the renowned result of Coifman--Rochberg--Weiss \cite{coifman1976factorization} and bypassed the use of classical Fourier transform as it is not applicable in the Bessel setting. Moreover, the Bloom type two weight commutator was also studied very recently in \cite{duong2021two}, via treating $(\R_+,|\cdot|, \dmz)$ as a space of homogeneous type again and using the weighted $L^p$ as in \eqref{classical Lpw}.
Thus, along the line of Coifman--Rochberg--Weiss \cite{coifman1976factorization} and the recent progression, the second natural question that we focus is on the weighted $L^p$ estimate (via \eqref{new Lpw}) of the commutator 
$[b, R_\lambda]$ under this new Muckenhoupt-type weights in the Bessel Setting.

\medskip
\noindent {\bf Question 2:}
Can we fully characterise the $L^p(\R_+,w)$-boundedness (via \eqref{new Lpw}) of the commutator $[b, R_\lambda]$ for $w\in A_{p,\lambda}$? That is, for $p\in (1,\infty)$ and for a suitable $\rm{BMO}$ space associated with $\tbz$,
$$
\|[b,R_\lambda]\|_{L^p(\R_+,w)\to L^p(\R_+,w)} \sim \|b\|_{\rm BMO}\quad ?
$$

\medskip
In this paper, we give full-fledged answers to these questions, respectively.



\subsection{Statement of main results}

Our first observation on the inner structure of the Anderson--Kerman weight class is as follows.
Regarding the definition of $A_{p,\lambda}$ as in \eqref{Ap lambda} (especially from the right-hand side of \eqref{Ap lambda}), we see that
the underlying measure is in fact $d\nu_{\lambda}(x):=x^{2\lambda +1}dx$. Hence, a natural definition of maximal function is associated with $d\nu_{\lambda}(x)$ but not the usual one associated with $dm_{\lambda}(x)= x^{2\lambda}dx$. That is,
\begin{defn}[$\lambda$-Maximal function]\label{defmax}
\[
\mathcal{M}_{\lambda}(x):=\sup_{B\ni x}\frac{1}{\nu_{\lambda}(B)}\int_{B}|f(y)|d\nu_{\lambda}(y).
\]
\end{defn}
To address {\bf Question 1}, we may ask whether 
$\mathcal{M}_{\lambda}$ is bounded on $L^p(\R_+,w)$ {\bf if and only if } $w\in A_{p,\lambda}$.

Surprisingly, it turns out that the right class for the maximal function $M_{\lambda}$ to be bounded on $L^p(\R_+,w)$ 
is another $\widetilde{A}_{p,\lambda}$ defined as below. 
 \begin{defn}\label{defaplt}
  For a non-negative measurable function $w$, $w\in \widetilde{A}_{p,\lambda}$ if there exists a constant $C>0$ such that for all intervals $B\subset \R_+$ we have
  $$
\bigg(\frac{1}{\nu_\lambda(B)}\int_B w(t)dt\bigg)\bigg(\frac{1}{\nu_\lambda(B)}\int_B t^{
(2\lambda+1) p'}w(t)^{-\frac{1}{p-1}}dt\bigg)^{p-1}<C.
$$
 \end{defn}
 
 We now give a positive answer to {\bf Question 1}.

 
 \begin{thm}[$L^p(\R_+,w)$-Boundedness for $\mathcal{M}_\lambda$]\label{thmsbm}~\\
 Let $w$ be a non-negative measurable function. Then for all $1<p<\infty$
$$
w\in \widetilde{A}_{p,\lambda} \Longleftrightarrow
\|\mathcal{M}_{\lambda} f\|_{L^p(\R_+,w)}\lesssim\|f\|_{L^p(\R_+,w)},
$$
\end{thm}

The key to obtaining the above result is the following observation (as well as the related approach and techniques): 
 The following inequality
\begin{align*}
    \int_{\R_+} (\mathcal{M}_{\lambda}f)(t)^{p}w(t)dt&=\int_{\R_+}(\mathcal{M}_{\lambda}f)(t)^{p}\frac{w(t)}{\nu_{\lambda}(t)}d\nu_{\lambda}(t)\\
    &\lesssim \int_{\R_+} |f(t)|^{p}\frac{w(t)}{\nu_{\lambda}(t)}d\nu_{\lambda}(t)
    = \int_{\R_+} |f(t)|^{p}w(t)dt
\end{align*}
holds 
if and only if the weight $\frac{w(t)}{\nu_{\lambda} (t)}$ satisfies the classical $A_{p}$ condition with respect to the doubling measure $\nu_{\lambda}$ (when $\lambda>-1/2$). That is,
 there exists a constant $C>0$ such that for all intervals $B\subset\R_+$
$$
\bigg(\frac{1}{\nu_\lambda(B)}\int_B \frac{w(t)}{\nu_{\lambda(t)}}d\nu_{\lambda}(t)\bigg)\bigg(\frac{1}{\nu_\lambda(B)}\int_B \Big(\frac{w(t)}{\nu_{\lambda} (t)}\Big)^{-\frac{1}{p-1}}d\nu_{\lambda}(t)\bigg)^{p-1}<C.
$$
Notice that
\[
\bigg(\frac{w(t)}{\nu_{\lambda} (t)}\bigg)^{-\frac{1}{p-1}}\cdot \nu_{\lambda}(t)=\nu_{\lambda}(t)^{p'}w(t)^{-\frac{1}{p-1}}=t^{(2\lambda+1)p'}\cdot w(t)^{-\frac{1}{p-1}}.
\]
Equivalently, $w\in \widetilde{A}_{p,\lambda}$.

Now, as a natural and further question, one may ask: Do $A_{p,\lambda}$ and $\widetilde{A}_{p,\lambda}$ enjoy the same characterization for the boundedness of $\mathcal{M}_{\lambda}$? Or are they actually the same class but just have different expressions?

The answer is \textbf{No}, that is, $A_{p,\lambda}$ and $\widetilde{A}_{p,\lambda}$ are different classes of weights. We will also exploit the structures for these two classes, respectively. Roughly speaking, $A_{p,\lambda}$ is the suitable class associated with Bessel Riesz transform $R_\lambda$ as it captures the intrinsic behavior for the kernel of $R_\lambda$, while $\widetilde{A}_{p,\lambda}$ is the natural one associated with the maximal function $\mathcal{M}_\lambda$ since it can be viewed as a special case for the classic $A_{p}$ weight on homogeneous space.


\begin{prop}[Structure of $A_{p,\lambda}$]\label{propapl}~\\
Let $\lambda>-\frac{1}{2}$ and $\lambda\not=0$. Then for $w\in A_{p,\lambda}$, we have the following properties\\
$(1)$ $$[w]_{A_{p,\lambda}}:=\sup_{B\subset\R_+}\bigg(\frac{1}{\nu_\lambda (B)}\int_B t^p w(t)dt\bigg)\bigg(\frac{1}{\nu_\lambda (B)}\int_B t^{2\lambda p'} w(t)^{\frac{-1}{p-1}}dt\bigg)^{p-1}\geq 1.$$
$(2)$ Define the following measures 
\begin{align*}
d\mu(t):=t^p w(t)dt \quad\text{and}\quad d\sigma(t):=t^{2\lambda p'}{w(t)}^{\frac{-1}{p-1}}.
\end{align*}
Then both of $d\mu$ and $d\sigma$ are doubling measures on $\R_+$.  \\
$(3)$ {\rm(Dual Factor)}
$$
w\in A_{p,\lambda}\Longleftrightarrow t^{(2\lambda-1)p'} {w(t)}^{\frac{-1}{p-1}}\in A_{p',\lambda}.
$$
$(4)$ {\rm(Nesting Structure)}
\[
\forall \lambda>-\frac{1}{2}, \lambda\not=0, 1\leq p_1<p
_2<\infty \implies A_{p_1,\lambda}\subset A_{p_2,\lambda}, 
\]
and
\[
\forall p\in (1,\infty),\ \lambda_2>\lambda_1>-1/2,\ \lambda_2,\ \lambda_1\neq0 \implies A_{p,\lambda_1}\subset A_{p,\lambda_2}.
\]
\end{prop}

\begin{prop}[Relationships among $A_p$, $A_{p,\lambda}$ and
$\widetilde{A}_{p,\lambda}$ on $\R_+$]\label{proprelap}~\\
Let $w(t):=t^{\alpha}$ on $\R_+$. Then

$(1)$  
$
w\in A_p \Longleftrightarrow -1<\alpha<p-1.
$\\

$(2)$
$
w\in A_{p,\lambda}\Longleftrightarrow -1-p<\alpha<p-1+2\lambda.
$\\

$(3)$
$
w\in\widetilde{A}_{p,\lambda}\Longleftrightarrow -1<\alpha<p-1+(2\lambda+1)p.
$
\end{prop}

\medskip

 We now turn to the estimate of the commutator. Given a measurable function $b(x)$, we can define the commutator $[b,R_\lambda]$ of the Bessel Riesz transform by
$$
[b,R_\lambda](f)(x):=b(x)R_\lambda (f)(x)-R_\lambda (bf)(x).
$$

%

The boundedness of commutators has a long history and has many important applications, for instance, the composition of differential operators. It is a well-known fact that for general Calder\'on--Zyumund operators $T$, the boundedness of commutators $[b, T]$ is equivalent to $b \in {\rm BMO}$. Thus we now define the $\rm{BMO}$ functions in our setting which will characterize the boundedness of $[b, R_{\lambda} ]$ in $L^p(\mathbb R_+, w)$ for $ w \in A_{p, \lambda}$.

\begin{defn}[Sharp $\lambda$-maximal function]\label{defsharpm}~\\
Let $w$ be a non-negative, measurable function on $\mathbb{R_+}$, and $f\in L^1_{loc} (\mathbb{R_+},\nu_\lambda)$. Define the sharp $\lambda$-maximal function by
\begin{align*}
    \mathcal{M}^{\#}_\lambda f(x)= \sup_{B\ni x, B\subset \mathbb R_+} \frac{1}{\nu_\lambda(B)}\int_B|f(t)-f_{B,\lambda}|d\nu_\lambda(t),
\end{align*}
where
\begin{align}\label{fBlambda}
f_{B,\lambda}:=  \frac{1}{\nu_\lambda(B)}  \int_B f(t)\, d\nu_\lambda(t),\quad \text{and}  \quad B=(a,b)\subset\R_+.
\end{align}
We say that $f$ has $\lambda$-bounded mean oscillation if $\mathcal{M}^{\#}_\lambda f\in L^\infty.$
\end{defn}

\begin{rem}
Throughout the paper, we refer $\nu_\lambda (t)$ to be $t^{1+2\lambda}$ where $\lambda> \frac{-1}{2}$ and $\lambda\not=0$. 
\end{rem}

\begin{defn}[The space ${\rm BMO}_\lambda (\mathbb{R_+})$]\label{defbmo}~\\
The space of $\lambda$-bounded mean oscillation on $(\mathbb{R_+},d\nu_\lambda)$ is given by
$$
{\rm BMO}_\lambda (\mathbb{R_+}):=\{f\in L^1_{loc} (\mathbb{R_+},d\nu_\lambda):\mathcal{M}^{\#}_\lambda f\in L^\infty\},
$$
equipped with the norm
\begin{align}\label{fB}
\|f\|_{{\rm BMO}_\lambda}:=\|\mathcal{M}^{\#}_\lambda f\|_{L^\infty}.
\end{align}

\end{defn}

We now give a positive answer to {\bf Question 2}.
\begin{thm}\label{thmcom}
Suppose $b\in L^1_{loc} (\mathbb{R_+},d\nu_\lambda)$, $1<p<\infty$ and $w\in A_{p,\lambda}$. 

\noindent If $b\in {\rm BMO}_\lambda(\mathbb{R_+})$, then
$$
\|[b,R_\lambda]\|_{L^p(\mathbb{R_+},w)\to L^p(\mathbb{R_+},w)} \lesssim \|b\|_{BMO_\lambda(\R_+)}.
$$

\noindent Conversely,  if $\|[b,R_\lambda]\|_{L^p(\R_+,w)\to L^p(\R_+,w)}<\infty$, 
then  $b\in {\rm BMO}_\lambda (\R_+)$ with
$$
 \|b\|_{{\rm BMO}_\lambda(\R_+)}\lesssim \|[b,R_\lambda]\|_{L^p(\R_+,w)\to L^p(\R_+,w)}.
$$
\end{thm}

It is natural to ask what is the relationship between our $\rm{BMO}_\lambda(\R_+)$ and the well-known ${\rm BMO}_{\Delta_\lambda}(\R_+)$, which was used before
 \cite{duong2017factorization} for unweighted estimate of $[b,R_\lambda]$, defined by $$
{\rm BMO}_{\Delta_\lambda} (\mathbb{R_+}):=\{f\in L^1_{loc} (\mathbb{R_+},dm_\lambda): \|f\|_{{\rm BMO}_{\Delta_\lambda} (\mathbb{R_+})}<\infty\},
$$
equipped with the norm
$$
\|f\|_{{\rm BMO}_{\Delta_\lambda} (\mathbb{R_+})}:=\sup_{B\subset \R_+} \frac{1}{m_\lambda(B)}\int_B|f(x)-f_B|dm_\lambda(x),
$$
where 
$$f_B=\frac{1}{m_\lambda(B)}\int_B f(x)dm_\lambda(x) \quad \text{and}  \quad B=(a,b)\subset\R_+. $$

Indeed, they are equivalent in the norm sense.
\begin{prop}\label{propbmo} Let $\lambda>-1/2$ and $\lambda\neq 0$,
$    {\rm BMO}_{\Delta_\lambda} (\mathbb{R_+})$ coincides with ${\rm BMO}_\lambda(\R_+)$
and they have equivalent norms.
\end{prop}

\subsection{Organization of the paper}

In Section $2$, we recall some preliminary results, including the relation between $A_{p,\lambda}$ class and the Bessel Riesz transform, the John-Nirenberg inequality in the doubling measure setting. In Sections $3$ and $4$, we demonstrate the inner properties of $A_{p,\lambda}$ and $\widetilde{A}_{p,\lambda}$, respectively, and compare the relation between these two classes and the classical one $A_p$. In Sections $5$ and $6$, we derive some reverse structure of $A_{p, \lambda}$ and $\widetilde{A}_{p,\lambda}$, respectively. And, in Section $7$, we give a comprehensive characterization of the $L^p(\R_+,w)$-boundedness of the commutator $[b,R_\lambda]$.

\bigskip

\section{Preliminaries}

\subsection{$A_{p,\lambda}$ Class}
In \cite{andersen1981weighted}, a class of weight $A_{p,\lambda}$ is discovered by Andersen and Kerman who showed the following Theorem.

\begin{thm}[$L^p(\R_+,w)$-Boundedness for Bessel Riesz Transform $R_\lambda$]~\\ For all $1<p<\infty$, $\lambda>-\frac{1}{2},\neq 0$, the following statemennts are equivalent:\\
$(1).$ $$
w\in A_{p,\lambda}.
$$
$(2).$ $$
\|R_\lambda (f)\|_{L^p (\R_+,w)}\underset{p,\lambda}{\lesssim}\|f\|_{L^p (\R_+,w)}.
$$
$(3).$ $$
\|R_\lambda (f)\|_{L^{p,\infty} (\R_+,w)}\underset{p,\lambda}{\lesssim}\|f\|_{L^p (\R_+,w)}.
$$
\end{thm}
For the endpoint case, that is $p=1$, they also proved that
$$
w\in A_{1,\lambda}\Longleftrightarrow \|R_\lambda (f)\|_{L^{1,\infty}(\R_+,w)}\underset{\lambda}{\lesssim}\|f\|_{L^{1} (\R_+,w)}.
$$
It is worth pointing out that in their paper, one of the intriguing parts is that in the classical $A_p$ theorem, the $A_1$ class can be obtained from $A_p,\: p>1$ by letting $p \rightarrow 1^+$. However the $A_{1,\lambda}$ case is not the limiting of $A_{p,\lambda}$ as $ p \rightarrow 1^+$.
\bigskip

\subsection{Space of Bounded Mean Oscillation}
For the equivalency of some weighted $\rm{BMO}$ norms, we refer \cite{ho2011characterizations} to the readers. In the following, we first prove the John--Nirenberg inequality in our $\rm{BMO}$ class. The proof is essentially the same as the classical proof. For the sake of completeness of the paper, we give the details here. 
Define the average associated with the doubling measure $\nu$ by 
$$
f_{B,\nu}:=  \frac{1}{\nu (B)}  \int_B f(t) d\nu(t), \quad\text{for~each~interval}\quad B\subset\R_+,
$$
and 
$$
\|f\|_{{\rm BMO}_\nu}:=\|\mathcal{M}^{\#}_\nu f\|_{L^\infty},
$$
where
\begin{align*}
    \mathcal{M}^{\#}_\nu f(x)= \sup_{B\ni x, B\subset \mathbb R_+} \frac{1}{\nu (B)}\int_B|f(t)-f_{B,\nu}|d\nu (t)
\end{align*}
is the sharp maximal function with respect to the doubling measure $\nu.$
\bigskip

\begin{lem}[John--Nirenberg Inequality]~\\
    Let $\nu$ be a doubling measure on $\R_+$ with doubling constant $c>1$. Then we have for all interval $B\subset \R_+$ and all $t>0$
    \[
    \nu \left(\left\{ x\in B  : |f(x)-f_{B,\nu}|>t \right\}\right)\leq e\cdot\nu (B)e^{\frac{-At}{\|f\|_{{\rm BMO}_{\nu}}}}
    \]
    where $A=(e\cdot c)^{-1}$.
\end{lem}

\begin{proof}
    First, we normalize $\|f\|_{{\rm BMO}_{\nu}}=1$. Fix a $B=B^{(0)}$, we perform Calder\'on--Zygmund decomposition at height $\beta$ to $f-f_{B,\nu}$ on $B$ to get a collection of intervals $\{ B_{j}^{(1)}\}_{j}$. Then the process is repeated again to $f-f_{B_{j}^{(1)},\nu}$ for each $B_{j}^{(1)}$ to get a second collection $\{ B_{l}^{(1)}\}_{l}$. Continuing this process and let $\{ B_{j}^{(k)}\}_{j}$ be the collection of intervals in the $k$-th generation. Then we have the following properties
\begin{align*}
        &(1).\quad \beta <\frac{1}{\nu (B_{j}^{(k)})}\int_{B_{j}^{(k)}}|f(x)-f_{B_{j'}^{k-1},\nu}|d\nu (x)\leq c\cdot \beta,\\
        &(2). \quad |f_{B_{j}^{k},\nu}-f_{B_{j'}^{k-1},\nu}|\leq c\cdot \beta,\\
        &(3).\quad \sum_{j}\nu (B_{j}^{(k)})\leq \frac{1}{\beta} \sum_{j'}\nu (B_{j'}^{(k-1)}),\\
        &(4).\quad |f-f_{B_{j'}^{(k-1)},\nu}|\leq \beta \quad a.e. \quad on \quad B_{j'}^{(k-1)} \setminus\bigcup_{j}B_{j}^{(k)},
\end{align*}
where $B_{j'}^{k-1}$ is the parent of $B_{j}^{(k)}$.\\
From $(3)$, we immediately have
\[
\sum_{j}\nu (B_{j}^{(k)})\leq \beta^{-k}\nu (B).
\]
On the other hand, by $(3)$ and $(4)$, we have
\[
|f-f_{B,\nu}|\leq c\cdot k \beta\quad \text{on} \quad B \setminus \bigcup_{j}B_{j}^{(k)}.
\]
Hence, we have
\[
\left\{ x\in B :|f(x)-f_{B,,\nu}|>c\cdot k \beta \right\}\subset\bigcup_{j}B_{j}^{(k)}.
\]
For $t$ in the range $ck\beta<t\leq c(k+1)\beta$ for some $k>0$, we have
\begin{align*}
    \nu \left(\left\{ x\in B : |f(x)-f_{B,\nu}|>t \right\}\right)&\leq \nu \left( \left\{ x\in B :|f(x)-f_{B,,\nu}|>ck\beta\right\}\right)\\
    &\leq \sum_{j}\nu(B_{j}^{k})\\
    &\leq \lambda^{-k}\nu (B)=e^{-k\cdot\log \beta}\cdot\nu (B).
\end{align*}
Notice that $-k<1-\frac{t}{c\beta}$ and take $\beta=e$, then we have the desired result.
\end{proof}

\bigskip

\begin{cor}\label{cor exp}
    Suppose $f\in {\rm BMO}_{\nu}$. Then we have the following two results.\\
    $(1)$ There is a constant $C_s<+\infty$ such that for any interval $B\subset\R_+$, 
    $$
    |s|<(ce\|f\|_{{\rm BMO}_{\nu}})^{-1} \implies
    \frac{1}{\nu (B)}\int_{B}e^{s|f(x)-f_{B,\nu}|}d\nu \leq     C_s:=1+e\cdot \frac{A^{-1}\cdot|s|\cdot\|f\|_{{\rm BMO}_\nu}}{1-A^{-1}\cdot|s|\cdot\|f\|_{{\rm BMO}_\nu}}.
    $$
    $(2)$ For all $1<p<\infty,$
    $$ \sup_{B\subset\R_+}\bigg(\frac{1}{\nu (B)}\int_{B}e^{sf(x)}d\nu(x)\bigg)\bigg(\frac{1}{\nu (B)}\int_{B}e^{\frac{-1}{p-1}\cdot sf(x)}d\nu(x)\bigg)^{p-1}\leq C_s^{p}.
$$.
\end{cor}

\begin{proof}
    $(1)$ By layer-cake representation, we have the following identity
    \[
    \frac{1}{\nu (B)}\int_{B}e^{h}d\nu=1+\frac{1}{\nu (B)}\int_{0}^{\infty}e^{t}\nu \left(\left \{ x\in B : |h(x)|>t \right\}\right)dt.
    \]
    Then we plug in $h=s|f-f_{B,\nu}|$ to obtain 
    \begin{align*}
    \frac{1}{\nu (B)}\int_{B}e^{s|f-f_{B,\nu}|}d\nu&=1+\frac{1}{\nu (B)}\int_{0}^{\infty}e^{t}\nu \left(\left \{ x\in B : |s||f(x)-f_{B,\nu}|>t \right\}\right)dt\\
    &\leq 1+\frac{1}{\nu(B)}\int^\infty_0 e^t \cdot e\cdot\nu(B) e^{\frac{-At}{|s|\cdot\|f\|{{\rm BMO}_{\nu}}}}dt\\
    &=1+e\cdot\int^\infty_0 \exp\left\{t\left(1-\frac{1}{c e |s| \|f\|_{{\rm BMO}_\nu} }\right)\right\}dt\\
    &=1+e\cdot \frac{ce\cdot|s|\|f\|_{{\rm BMO}_\nu}}{1-ce\cdot|s|\cdot\|f\|_{{\rm BMO}_\nu}}<\infty.
    \end{align*}
    \\
    $(2)$ From above, we have for all interval $B\subset\R_+$ 
\[
|s|<\frac{A}{\|f\|_{{\rm BMO}_\nu}}
\implies
\frac{1}{\nu (B)}\int_{B}e^{s(f(x)-f_{B,\nu})}d\nu \leq C_s
\]
and
\[
\left|\frac{s}{p-1}\right|<\frac{A}{\|f\|_{{\rm BMO}_\nu}}
\implies    \left(\frac{1}{\nu (B)}\int_{B}e^{\frac{-1}{p-1} s(f(x)-f_{B,\nu})}d\nu\right)^{p-1} \leq C^{p-1}_s.
\]
    Multiplying them together, then we have
\[
|s|<\frac{A}{\|f\|_{{\rm BMO}_\nu}}\cdot \min\{1,p-1\}\implies
\sup_{B}\left(\frac{1}{\nu (B)}\int_{B}e^{sf(x)}d\nu(x)\right)\left(\frac{1}{\nu (B)}\int_{B}e^{\frac{-1}{p-1}sf(x)}d\nu\right)^{p-1}\leq C^{p}_s.
\]
\end{proof}

\bigskip

\section{Proof of Proposition \ref{propapl} and Properties of $\widetilde{A}_{p,\lambda}$}

Before proving the $L^p (\R_+,w)$-boundedness of maximal function $\cm_\lambda$, we first study the structure of $A_{p,\lambda}$, and $\widetilde{A}_{p,\lambda}$ and their relationships with the classical $A_p$ weights. For the $A_{p,\lambda}$ class, we define the norm of a  $A_{p,\lambda}$ weight $w$ to be
\begin{align}\label{aplcond norm}
[w]_{A_{p,\lambda}}&=\sup_{B\subset\R_+}\bigg(\frac{1}{\nu_\lambda (B)}\int_Bt^{p}w(t)dt\bigg)\cdot\bigg( \frac{1}{\nu_\lambda (B)}
\int_B t^{2\lambda p'}w(t)^{-\frac{1}{p-1}}dt\bigg)^{p-1},\nonumber 
\end{align}
and develop their natural structure analogous to $A_p$.
\bigskip

$(1)$
$$[w]_{A_{p,\lambda}}\geq 1. $$

\begin{proof}
\begin{align*}
1&=   \frac{1}{\nu_\lambda (B)}  \int_B t^{1+2\lambda}dt \\ 
&=   \frac{1}{\nu_\lambda (B)}  \int_B t\, w(t)^{1\over p}\, t^{2\lambda}\,w(t)^{-{1\over p}}dt \\ 
&\leq   \frac{1}{\nu_\lambda (B)} \bigg(\int_B t^pw(t)dt\bigg)^{1\over p} \bigg( \int_B t^{2\lambda p'}w(t)^{-{p'\over p}}dt  \bigg)^{1\over p'}\\
&\leq  \bigg( \frac{1}{\nu_\lambda (B)} \int_B t^pw(t)dt\bigg)^{1\over p} \bigg(  \frac{1}{\nu_\lambda (B)}\int_B t^{2\lambda p'}w(t)^{-{p'\over p}}dt  \bigg)^{1\over p'}\\
&\leq [w]_{A_{p,\lambda}}^{1\over p}.
\end{align*}
Thus,
we have
\begin{align*}
[w]_{A_{p,\lambda}}&\geq 1.
\end{align*}
\end{proof}

\medskip

$(2)$ The following two measures are both doubling on $\R_+$:
$$d\mu(t)= t^pw(t)dt, $$
and
$$d\sigma(t)= t^{2\lambda p'}w(t)^{-{p'\over p}}dt. $$

\begin{proof}
Recall that a new average of $f$ over interval $B=(a,b)$ as follows: 
\begin{align*}
f_{B,\lambda}:=   \frac{1}{\nu_\lambda (B)}  \int_{B} f(t)\, t^{1+2\lambda}\,dt,
\end{align*}
which implies that
\begin{align*}
(f_{B,\lambda})^p&\leq \left(  \frac{1}{\nu_\lambda (B)}  \int_B |f(t)|\ t^{1+2\lambda}\,dt \right)^p\\ 
&= \left(   \frac{1}{\nu_\lambda (B)}\int_B |f(t)|\ t\, w(t)^{1\over p}\cdot t^{2\lambda}\,w(t)^{-{1\over p}}dt \right)^p\\
&\leq \frac{1}{\nu_\lambda (B)^p}\left(\int_B |f(t)|^p t^pw(t)dt\right) \left( \int_B t^{2\lambda p'}w(t)^{-{p'\over p}}dt  \right)^{p\over p'} \\
&=   \left(\int_B |f(t)|^p t^pw(t)dt\right)\cdot \left(\int_B t^pw(t)dt\right)^{-1}\\
&\qquad\times \left( \frac{1}{\nu_\lambda (B)}\int_B t^pw(t)dt\right)\cdot \left( \frac{1}{\nu_\lambda (B)}  \int_B t^{2\lambda p'}w(t)^{-{p'\over p}}dt \right)^{p\over p'} \\
&\leq\left(\int_B |f(t)|^p t^pw(t)dt\right)\cdot\left(\int_B t^pw(t)dt\right)^{-1} \cdot [w]_{A_{p,\lambda}}.
\end{align*}
Thus, if we take $f = \chi_{B}$ and put $\eta B$ in the place of $B$ in the above estimates with $\eta>1$, we have  

\begin{align*}
\int_{\eta B} t^pw(t)dt \leq C_p [w]_{A_{p,\lambda}} \eta^{(2\lambda+2) p} \int_{B} t^pw(t)dt.
\end{align*}
This shows that 
$d\mu(t)= t^pw(t)dt $ is a doubling measure on $\mathbb R_+$. One can show that $d\sigma$ is a doubling measure as well.
\end{proof}

\medskip

$(3)$ (Dual Factor)
$$w\in A_{p,\lambda}\Longleftrightarrow t^{(2\lambda-1)p'}w^{-1\over p-1}\in A_{p',\lambda}.$$

\begin{proof}
Assume that $w\in A_{p,\lambda}$, then we have for some $ C>0$ such that for all interval $B:=(a,b)\subset \R_+$,
\[
\left( \frac{1}{\nu_\lambda (B)} \int_B t^pw(t)dt\right)\cdot\left(  \frac{1}{\nu_\lambda (B)}\int_B t^{2\lambda p'}w(t)^{-{p'\over p}}dt \right)^{p\over p'}<C.
\]
On the other hand,
\begin{align*}
\left(\int_{B}t^{p'}t^{(2\lambda-1)p'}w(t)^{-{1\over p-1}}dt\right)\cdot\left(\int_{B}t^{2\lambda p}(t^{(2\lambda-1)p'}w(t)^{-{1\over p-1}})^{-\frac{1}{p'-1}}dt\right)^{p'-1},
\end{align*}
which is equal to
\begin{align*}
\left(\int_{B}t^{2\lambda p'}w(t)^{-\frac{1}{p-1}}dt\right)\cdot\left(\int_{B} t^p w(t)dt\right)^{1\over p-1}.
\end{align*}
Thus,
\begin{align*}
\left(\frac{1}{\nu_{\lambda} (B)}\int_{B}t^{p'}t^{(2\lambda-1)p'}w(t)^{-{1\over p-1}}dt\right)\cdot\left(\frac{1}{\nu_{\lambda}(B)}\int_{B}t^{2\lambda p}(t^{(2\lambda-1)p'}w(t)^{-{1\over p-1}})^{-\frac{1}{p'-1}}dt\right)^{p'-1}
\leq C.
\end{align*}
The other direction is similar.
\end{proof}

\medskip

 $(4)$(Nesting structure)
 \[
\forall \lambda >-\frac{1}{2},\neq 0, 1\leq p_1<p
_2<\infty \implies A_{p_1,\lambda}\subset A_{p_2,\lambda}, 
\]
and
\[
\forall p\in (1,\infty),\lambda_2>\lambda_1>-\frac{1}{2},\neq 0 \implies A_{p,\lambda_1}\subset A_{p,\lambda_2}.
\]

\begin{proof}
 Suppose $w\in A_{p_1,\lambda}$, then for all $B:=(a,b)\subset\R_+$,
 \begin{align*}
&\left(\int_{a}^{b}t^{p_2}w(t)dt\right)\cdot\left(\int_{a}^{b}t^{2\lambda {p_2}'}w(t)^{-\frac{1}{{p_2}-1}}dt\right)^{p_2-1}\\
&\leq b^{{p_2}-{p_1}}\left(\int_{a}^{b}t^{p_1}w(t)dt\right)\cdot\left(\int_{a}^{b}t^{2\lambda {p_2}'}w(t)^{-\frac{1}{{p_2}-1}}dt\right)^{p_2-1}\\
&= b^{{p_2}-{p_1}}\left(\int_{a}^{b}t^{p_1}w(t)dt\right)\cdot\left(\int_{a}^{b}\left(t^{2\lambda {p_1\over p_2-1}}w(t)^{-\frac{1}{{p_2}-1}}\right)t^{2\lambda{{p_2-p_1}\over {p_2-1}}}dt\right)^{p_2-1},
\end{align*}
Applying Holder's inequality with exponent ${p_2-1}\over{p_1-1}$ and ${p_2-1}\over{p_2-p_1}$ to the second term, and hence
\begin{align*}
\left(\int_{a}^{b}t^{p_2}w(t)dt\right)\cdot\left(\int_{a}^{b}t^{2\lambda {p_2}'}w(t)^{-\frac{1}{{p_2}-1}}dt\right)^{p_2-1}
&\leq Cb^{{p_2}-{p_1}}(b^{2(\lambda+1)}-a^{2(\lambda +1)})^{p_1-1}\\
&\cdot (b^{2\lambda+1}-a^{2\lambda+1})^{p_2-p_1}\\
&\leq C(b^{2(\lambda+1)}-a^{2(\lambda +1)})^{p_2},
\end{align*}
or equivalently
$$
[w]_{A_{p_2,\lambda}}\leq [w]_{A_{p_1,\lambda}}.
$$
\bigskip
For the second part, we let $w\in A_{p,\lambda_1}$, then
\begin{align*}
\left(\int_{a}^{b}t^{p}w(t)dt\right)\cdot\left(\int_{a}^{b}t^{2\lambda_2 p'}w(t)^{-\frac{1}{p-1}}dt\right)^{p-1}
&\leq\left(\int_{a}^{b}t^{p}w(t)dt\right)\cdot\left(\int_{a}^{b}t^{2\lambda_1 p'}w(t)^{-\frac{1}{p-1}}dt\right)^{p-1}\\
&\cdot  (b^{2{p'}(\lambda_2-\lambda_1)})^{p-1}\\
&\leq C (b^{2(\lambda_1+1)}-a^{2(\lambda_1+1)})^{p}(b^{2{p'}(\lambda_2-\lambda_1)})^{p-1}\\ 
&\leq C(b^{2(\lambda_2+1)}-a^{2(\lambda_2+1)})^{p}.
\end{align*}
Thus, $w\in A_{p,\lambda_2}$.
\end{proof}

\medskip

We also have analogous structures for $\widetilde{A}_{p,\lambda}.$ We first recall the definition of $\widetilde{A}_{p,\lambda}$.
\begin{defn}[$\widetilde{A}_{p,\lambda}$ condition]~\\
We say that $w\in \widetilde{A}_{p,\lambda}$ if there exists $C>0$ such that for all interval $B:=(a,b)\subset\R_+$
$$
\left({1\over \nu_\lambda(B)}\int_B w(t)dt \right)\cdot\left({1\over \nu_\lambda(B)}\int_{B}t^{(2\lambda+1) p'}w(t)^{-\frac{1}{p-1}}dt\right)^{p-1}\leq C,
$$
or equivalently
$$
\left({1\over \nu_\lambda(B)}\int_{B}\frac{w(t)}{\nu_\lambda(t)}d\nu_\lambda (t) \right)\cdot\left({1\over \nu_\lambda(B)}\int_{B} \left(\frac{w(t)}{\nu_\lambda(t)}\right)^{-\frac{1}{p-1}}d\nu_\lambda (t)\right)^{p-1}\leq C.
$$
\end{defn}

\medskip

\begin{rem}
 In general, for maximal operator $\M_{\nu}$, where $\nu$ is a doubling measure, a classical result by Calder\'on \cite{calderon1976inequalities} gives a characterization on weight $w$ such that $\M_{\nu}$ is bounded on $L^{p}(w, d\nu)$, that is
\[
\sup_{B}\left(\frac{1}{\nu (B)}\int_{B}w(x)d\nu(x) \right)\cdot
\left(\frac{1}{\nu (B)}\int_{B}w(x)^{-\frac{1}{p-1}}d\nu(x)\right)^{p-1}\lesssim 1.
\]
\end{rem}

\medskip

The proof of the following proposition is similar to the proof for $A_{p, \lambda}$ so that we only state the results and skip the proof.

\begin{prop}[Structure of $\widetilde{A}_{p,\lambda}$]~\\
$(1)$
$$[w]_{\widetilde{A}_{p,\lambda}}\geq 1.$$
$(2)$ If $w\in \widetilde{A}_{p,\lambda}$, then $w$ and $t^{(2\lambda+1)p'}w(t)^{\frac{-1}{p-1}}$ are both doubling measures on $\R_+$.\\
$(3)$ (Dual Factor)
$$
w\in \widetilde{A}_{p,\lambda}\Longleftrightarrow
t^{(2\lambda+1)p^{'}}w(t)^{\frac{-1}{p-1}}\in \widetilde{A}_{p',\lambda}.$$
$(4)$(Nesting structure)
 \[
\forall \lambda >-\frac{1}{2},\neq 0, 1\leq p_1<p
_2<\infty \implies \widetilde{A}_{p_1,\lambda}\subset \widetilde{A}_{p_2,\lambda}, 
\]
and
\[
\forall p\in (1,\infty),\lambda_2>\lambda_1>-\frac{1}{2},\neq 0 \implies \widetilde{A}_{p,\lambda_1}\subset \widetilde{A}_{p,\lambda_2}.
\]
\end{prop}

\bigskip
\bigskip

\section{Proof of Proposition \ref{proprelap}}
The main point of this proposition is to show that these three classes of weights, $A_p$, $A_{p,\lambda}$ and $\widetilde{A}_{p,\lambda}$ in general are different and especially $A_{p,\lambda}$ and $\widetilde{A}_{p,\lambda}$ are not contained in each other. Among explicit examples, we consider the simple but fundamental power weights.
Let $w(x)=|x|^{\alpha}$ for some $\alpha$.\\

$(1)$
\[
w\in A_{p} \Longleftrightarrow -1<\alpha <p-1.
\]
First, we need $ \alpha > -1$ to make $w$ to be doubling. Now given a ball $B=B(x_{0}, R)$, where $x_{0}$ is the center and $R$ is the radius, we consider the following two cases:
\begin{itemize}
    \item \textbf{Case(1)}:\,$|x_{0}|$ dominates.
\end{itemize}
In the first case, for those $|x-x_{0}|<R$, we have $|x|\sim |x_{0}|$, hence
\[
\left(\frac{1}{|B|}\int |x|^{\alpha}dx\right)\left(\frac{1}{|B|}\int |x|^{{-\alpha}\frac{p'}{p}}\right)^{p-1}\sim |x_{0}|^{\alpha}|x_{0}|^{-\alpha } \sim 1.
\]
\begin{itemize}
    \item \textbf{Case(2)}:\,$R$ dominates.
\end{itemize}
In the second case and by the doubling property, we may replace $B(x_{0}, R)$ by $B(0,5R)$ and compute the following
\[
\left(\frac{1}{R}\int_{-R}^{R}x^{\alpha}dx\right)\left(\frac{1}{R}\int_{-R}^{R}x^{-\alpha \frac{p'}{p}}dx\right)^{p-1}.
\]
If we want the above term to be finite uniformly in $R$, we only need to check both terms are integrable near $0$. Hence this gives the following conditions which give the range for $\alpha$.
\[
\alpha >-1,\quad -\alpha\frac{p'}{p}>-1.
\]

\bigskip

$(2)$
 \[
w\in A_{p,\lambda} \Longleftrightarrow -1-p<\alpha <p-1+2\lambda.
\]
The computations are similar to the above. We simply compute the following 
\[
\left(\frac{1}{R}\int_{-R}^{R}x^{\alpha +p}dx\right)\left(\frac{1}{R}\int_{-R}^{R}x^{2\lambda p'-\alpha \frac{p'}{p}}dx\right)^{p-1}
\]
that gives the range for $\alpha$ which is
\[
\alpha +p>-1, \quad 2\lambda p'-\alpha \frac{p'}{p}>-1.
\]

\bigskip

$(3)$

\[
w\in \widetilde{A}_{p,\lambda} \Longleftrightarrow -1<\alpha <p-1+(2\lambda+1)p.
\]
Again, we compute the following 
\[
\left(\frac{1}{R}\int_{-R}^{R}x^{\alpha }dx\right)\left(\frac{1}{R}\int_{-R}^{R}x^{(2\lambda+1) p'-\alpha \frac{p'}{p}}dx\right)^{p-1},
\]
which gives
\[
\alpha>-1, \quad (2\lambda+1) p'-\alpha \frac{p'}{p}>-1.
\]

\bigskip

Moreover, these three classes coincide locally. Precisely, let us define a local version as the following.
\begin{defn}[$A^{\text{loc}}_{p,\lambda}$, $\widetilde{A}^{\text{loc}}_{p,\lambda}$ and $A^{\text{loc}}_{p,+}$]~\\
Let $1<p<\infty$ and $\lambda>-\frac{1}{2},\neq 0. $
\\
$(1)$ 
$w\in A^{\text{loc}}_{p,\lambda}$ if there exist $C,k=k(w)>0$ such that for all $B:=(a,b)\subset (0,\infty)$ with $b\leq ka$,
$$\bigg({1\over \nu_\lambda(B)}\int_B t^p w(t)dt\bigg)\bigg({1\over \nu_\lambda(B)}\int_{B} t^{2\lambda p'}w(t)^{-\frac{1}{p-1}}dt\bigg)^{p-1}\leq C.
$$
$(2)$ 
$w\in \widetilde{A}^{\text{loc}}_{p,\lambda}$ if there exist $C,k=k(w)>0$ such that for all $B:=(a,b)\subset (0,\infty)$ with $b\leq ka$,
$$\bigg({1\over \nu_\lambda(B)}\int_{B} w(t)dt\bigg)\bigg({1\over \nu_\lambda(B)}\int_{B} t^{(2\lambda+1) p'}w(t)^{-\frac{1}{p-1}}dt\bigg)^{p-1}\leq C.
$$
$(3)$
$w\in A^{\text{loc}}_{p,+}$ if there exist $C,k=k(w)>0$ such that for all $(a,b)\subset (0,\infty)$ with $b\leq ka$,
$$\bigg({1\over {b-a}}\int_{a}^{b} w(t)dt\bigg)\bigg({1\over {b-a}}\int_{a}^{b} w(t)^{-\frac{1}{p-1}}dt\bigg)^{p-1}\leq C.
$$
\end{defn}

\bigskip

For $p\in(1,\infty)$ we have
    $$
A^{\text{loc}}_{p,+}=\widetilde{A}^{\text{loc}}_{p,\lambda}=A^{\text{loc}}_{p.\lambda}.
$$

Indeed, we observe that for all $0<a<b<\infty$,
$$
a^p\leq {{\int_{a}^{b}t^pw(t)dt}\over{\int_{a}^{b}w(t)dt}} \leq b^p,\quad ({1\over b})^p\leq {{\int_{a}^{b}w(t)dt}\over{\int_{a}^{b}t^pw(t)dt}} \leq ({1\over a})^p
$$
and
$$
{1\over b^p}
\leq{\left(\int_{a}^{b}t^{2\lambda p'}w(t)^{-\frac{1}{p-1}}dt\right)^{p-1}\over{\left(\int_{a}^{b}t^{(2\lambda+1) p'}w(t)^{-\frac{1}{p-1}}dt\right)^{p-1}}}
\leq {1\over a^p}, 
\quad 
a^p
\leq{\left(\int_{a}^{b}t^{(2\lambda+1) p'}w(t)^{-\frac{1}{p-1}}dt\right)^{p-1}\over{\left(\int_{a}^{b}t^{2\lambda p'}w(t)^{-\frac{1}{p-1}}dt\right)^{p-1}}}
\leq b^p.
$$
Therefore from the assumption we immediately have
$$\widetilde{A}^{\text{loc}}_{p,\lambda}=A^{\text{loc}}_{p,\lambda}.
$$
Likewise, we also have
$$
A^{\text{loc}}_{p,+}=\widetilde{A}^{\text{loc}}_{p,\lambda}.
$$

It might be worth pointing out that these simple properties for a local version of weights can be helpful when we study the boundedness of operators since we have many well-established results for the classical $A_p$ weights. 

\bigskip

\section{Proof of Weak-Type Estimate on $\M_ \lambda$}
Our ultimate goal for the maximal function $\M_\lambda$ is to prove the weighted strong-type boundedness. Before doing so, we show the weak-type boundedness first.

\begin{proof}
 To begin with, we need to show that $$ w\in \widetilde{A}_{p,\lambda} \Longleftrightarrow
 w(\{x\in\mathbb{R_+}:\mathcal M_{\lambda}f(x)>\alpha\})\leq \frac{C}{\alpha^p}\int_{\mathbb{R_+}}|f(x)|^pw(x)dx, \forall \alpha>0.
 $$
 $(\Longleftarrow)$
Let $B:=(a,b)\subseteq\mathbb{R_+}$ be an interval such that
$$
\int_B|f(x)|dx>0.
$$
For all $0<\alpha<|f|_{B,\lambda}$, we have 
$$
B\subset\{x\in\mathbb{R_+}:\cm_\lambda(f\,\mathbbm{1}_B)(x)>\alpha\}
\implies w(B)\leq \frac{C}{\alpha^p}\cdot\int_B|f(x)|^p w(x)dx,
$$
and hence
\begin{align*}
   w(B)\cdot\left(\frac{1}{\nu_\lambda(B)}\int_B |f(x)| x^{2\lambda+1}dx\right)^p\leq C\int_B|f(x)|^pw(x) dx.
\end{align*}
 Take $f(x)=w(x)^{1-p^{'}}x^{{2\lambda+1}\over {p-1}}$, then we have
 $$
 \left({1\over \nu_\lambda((a,b))}\int_{a}^{b}w(t)dt\right)\left({1\over \nu_\lambda((a,b))}\int_{a}^{b}t^{(2\lambda+1) p'}w(t)^{-\frac{1}{p-1}}dt\right)^{p-1}\leq C.
 $$
 $(\Longrightarrow)$
 Conversely,
\begin{align*}
\frac{1}{\nu_\lambda(B)}\int_B |f(x)| x^{2\lambda+1}dx&\leq \left(\frac{1}{\nu_\lambda(B)}\int_B|f(x)|^pw(x)dx\right)^{1\over{p}}\left(\frac{1}{\nu_\lambda(B)}\int_B x^{{p'}(2\lambda+1)} w(x)^{\frac{-1}{p-1}}dx\right)^{1\over{p'}}\\
&\leq C\left(\frac{1}{\nu_\lambda(B)}\int_B|f(x)|^pw(x)dx\right)^{1\over{p}}\left(\frac{w(B)}{\nu_\lambda (B)}\right)^{-1\over{p}},
\end{align*}
the second inequality holds because of $w \in \widetilde{A}_{p,\lambda},
$
and hence
$$
w(B)\cdot\left({1\over{\nu_\lambda(B)}}\int_B |f(x)| x^{2\lambda+1}dx\right)^p\leq C\int_B |f(x)|^p w(x)dx.
$$
By the classical covering method ($d\nu_\lambda$ is a doubling measure), we have
\begin{align*}  
w(\{x\in\mathbb{R}_+:\cm_{\lambda}f(x)>\alpha\}) &\leq C\sum_{B}w(B)\\
&\leq C\sum_{B} (|f|_{B,\lambda})^{-p}\int_{B}|f(x)|^{p} w(x)dx\\
&\leq {C\over \alpha^p} \int_{\mathbb{R}_+}|f(x)|^p w(x)dx,
\end{align*}
where the first inequality follows from that $w$ is doubling.
\end{proof}

Next, to prove the strong-type boundedness, we need the reverse structure of $\widetilde{A}_{p,\lambda}.$

\bigskip

\section{Dyadic Maximal $\M^{d}_\lambda$, Reverse H\"older Inequality on $\widetilde{A}_{p,\lambda}$, and Proof of Theorem \ref{thmsbm}}
In this and the next section, we investigate the reverse structure by following the classical method. We prove the Calder\'on--Zygmund decomposition with respect to the doubling measure $\nu_\lambda$, and use this decomposition to show the reverse type inequality. 
\begin{defn}[Dyadic Maximal Function]~\\
Let $\mathcal{D}_k$ be a family of dyadic intervals on $\mathbb{R_+}$ with side length $2^{-k}$, and $f\in L^1_{loc}(\mathbb{R_+},\nu_\lambda)$, the dyadic maximal function of $f$ is defined by
$$
\mathcal{M}^d_\lambda f(x):=\sup_{k} |\mathbb{E}_{k,\lambda}f(x)|,
$$
where
$$
\mathbb{E}_{k,\lambda}f(x):=\sum_{B\in\mathcal{D}_k}\bigg(\frac{1}{\nu_\lambda (B)}\int_B |f|d\nu_\lambda\bigg)\mathbbm{1}_B (x).
$$
\end{defn}

\begin{lem}[Stopping Time Lemma]~\\
The dyadic maximal function is of weak $(1,1)$, that is
$$
\|\mathcal M^d_\lambda f\|_{L^{1,\infty}(R_+,\nu_\lambda)} \leq C \|f\|_{L^1(R_+,\nu_\lambda)}.
$$
\end{lem}

\bigskip

\begin{proof}
Given any $\alpha>0$, we consider the set $$\Omega_k:=\left\{\mathbb{E}_{k,\lambda}f>\alpha,
\mathbb{E}_{i,\lambda}f\leq\alpha, \forall i< k
\right\},\quad \forall k.$$
Then the level set of dyadic maximal function can be decomposed by
$$
\{\mathcal M^d_\lambda f>\alpha\}=\bigsqcup_{k}\Omega_k,
$$
and hence
\begin{align*}
\nu_\lambda(\{\mathcal M^d_\lambda f>\alpha\})=\sum_{k}\nu_\lambda (\Omega_k)&\leq \frac{1}{\alpha}\sum_{k}\int_{\Omega_k}\mathbb{E}_{k,\lambda}f(x)d\nu_\lambda\\
&=\frac{1}{\alpha}\sum_{k}\int_{\Omega_k}|f|d\nu_\lambda.
\end{align*}
Thus, the dyadic maximal function is of weak $(1,1)$.
\end{proof}
By the stopping time lemma above, we have the Calder\'on--Zygmund decomposition with the doubling constant $4^{\lambda+1}$ in this setting:
\begin{thm}[Calder\'on--Zygmund Decomposition]~\\
For all $f\in L^1(\mathbb{R_+},\nu_\lambda)$ and let $\alpha>0$ be given. Then there exists a collection of disjoint dyadic intervals $B_k$ such that
$$
|f(x)|\leq \alpha,\quad\forall x\notin\bigcup_k B_k;\quad \nu_\lambda\bigg(\bigcup_k B_k\bigg)\leq\frac{1}{\alpha}\|f\|_{L^1(R_+,\nu_\lambda)}
$$
and
$$
\alpha<\frac{1}{\nu_\lambda(B_k)}\int_{B_k}|f|d\nu_\lambda\leq 4^{\lambda+1}\alpha.
$$
\end{thm}

\begin{thm}[Reverse Type Inequality]~\\
Let $w\in \widetilde{A}_{p,\lambda}$, $1< p <\infty$. Then there exists  $\varepsilon>0$ such that for any interval $B\subset\mathbb{R_+}$,
$$
\bigg(\frac{1}{\nu_\lambda(B)}\int_B (w\nu^{-1}_\lambda)^{1+\varepsilon}\,d\nu_\lambda\bigg )^{1/(1+\varepsilon)}\underset{p,w}{\lesssim}\frac{1}{\nu_\lambda(B)}\int_B w \nu^{-1}_\lambda d\nu_\lambda. 
$$
\end{thm}

\bigskip

To prove this result, we first need to prove the following lemma.
\begin{lem}~\\
 Let $w\in \widetilde{A}_{p,\lambda}$, $1< p<\infty$. Then for every $\alpha\in(0,1)$, there exists $\beta\in(0,1)$, such that for any interval $B$ and $S\subset B$ 
 $$\nu_\lambda(S)\leq \alpha \nu_\lambda(B)\implies w(S)=\int_S w(\nu_\lambda)^{-1}d\nu_\lambda\leq\beta \int_B w(\nu_\lambda)^{-1}d\nu_\lambda=\beta w(B).
 $$
\end{lem}

\bigskip

\begin{proof}
For $w\in \widetilde{A}_{p,\lambda}$, the inequality always holds:
\begin{align*}
    w(B)\left(\frac{1}{\nu_\lambda(B)}\int_B |f|d\nu_\lambda\right)^p\leq C\cdot \int_B |f|^p dw,
\end{align*}
which implies that
$$
    w(B)\left(\frac{\nu_\lambda(S)}{\nu_\lambda(B)}\right)^p\leq C\cdot w(S),
$$
If we replace $S$ by $B \setminus S$ we get
\begin{align*}
w(B)\left(1-\frac{\nu_\lambda(S)}{\nu_\lambda(B)}\right)^p\leq C\cdot(w(B)-w(S)).
\end{align*}
Since $\nu_\lambda (S)\leq \alpha \nu_\lambda (B)$, then
\begin{align*}
 w(B)(1-\alpha)^p\leq C\left(w(B)-w(S) \right)\Longleftrightarrow w(S)\leq \frac{C-(1-\alpha)^p}{C}w(B), 
\end{align*}
which gives us the desired result with $\beta = 1-C^{-1}(1-\alpha)^p\in(0,1)$.
\end{proof}

\bigskip

\begin{proof}(Reverse Type Inequality)\,
Fix an interval $B\subset\mathbb{R_+}$ and consider an increasing sequence 
$$\frac{1}{\nu_\lambda(B)}\int_B w(\nu_\lambda)^{-1}d\nu_\lambda: = \alpha_0<\alpha_1<\cdots<\alpha_k<\cdots.
$$
By Calder\'on--Zygmund decomposition, for all $k$, there exists $\{B^k_m\}_m$ such that
\begin{align*}
w(x)\nu_\lambda(x)^{-1}\leq \alpha_k,\ \forall x\notin \bigcup_m B^k_m;\quad \text{and}\quad \alpha_k<\frac{1}{\nu_\lambda(B^k_m)}\int_{B^k_m}w\nu^{-1}_\lambda d\nu_\lambda\leq 4^{\lambda+1}\alpha_k.
\end{align*}
It is clear from the construction that $\Omega_{k+1}\subset\Omega_{k}$, where $\Omega_k:=\bigcup_m B^k_m.$ If we fix $B^k_m$ from the Calder\'on--Zygmund decomposition at height $\alpha_k$, then $B^k_m\cap \Omega_{k+1}$ is the union of cubes $B^{k+1}_n$, with $n\in\Gamma_m$ from the decomposition at height $\alpha_{k+1}$. Therefore,
\begin{align*}
\nu_\lambda\left(B^k_m\cap \Omega_{k+1}\right)=\sum_{n\in\Gamma_m}\nu_\lambda(B^{k+1}_n)
&\leq \frac{1}{\alpha_{k+1}}\sum_{n\in\Gamma_m}\int_{B^{k+1}_n}w\nu^{-1}_\lambda d\nu_\lambda \\
&\leq \frac{1}{\alpha_{k+1}}\int_{B^k_m}w \nu^{-1}_\lambda d\nu_\lambda\leq \frac{4^{\lambda+1}\alpha_k}{\alpha_{k+1}}\nu_\lambda(B^K_m).
\end{align*}
Fix $\alpha\in(0,1)$ and choose the $\alpha_k$'s so that $$\frac{4^{\lambda+1}\alpha_k}{\alpha_{k+1}} = \alpha\implies \nu_\lambda(\Omega_{k+1})\leq\alpha\nu_\lambda(\Omega_k),\quad \forall k,
$$
and hence there exists $\beta\in(0,1)$ such that
\begin{align*}
\int_{\Omega_{k+1}}w\nu^{-1}_\lambda d\nu_\lambda\leq \beta\int_{\Omega_k}w\nu^{-1}_\lambda d\nu_\lambda\leq\cdots\leq\beta^k \int_{\Omega_0}w \nu^{-1}_\lambda d\nu_\lambda.
\end{align*}
Therefore,
\begin{align*}
\frac{1}{\nu_\lambda(B)}\int_B (w\nu^{-1}_\lambda)^{1+\varepsilon}d\nu_\lambda&=\frac{1}{\nu_\lambda(B)}\int_{B\setminus\Omega_0}(w\nu^{-1}_\lambda)^{1+\varepsilon}d\nu_\lambda+\frac{1}{\nu_\lambda(B)}\sum_{k=1}^\infty\int_{\Omega_k\setminus\Omega_{k+1}}(w\nu^{-1}_\lambda)^{1+\varepsilon}d\nu_\lambda\\
&\leq \alpha_0^\epsilon\frac{1}{\nu_\lambda(B)}\int_B w\nu^{-1}_\lambda d\nu_\lambda +\frac{1}{\nu_\lambda(B)}\sum_{k=1}^\infty\alpha_{k+1}^\varepsilon\cdot \int_{\Omega_k}w\nu^{-1}_\lambda d\nu_\lambda\\
&=\alpha_0^\epsilon\cdot\frac{1}{\nu_\lambda(B)}\int_B w\nu^{-1}_\lambda d\nu_\lambda+\frac{1}{\nu_\lambda(B)}\sum_{k=1}^\infty \alpha_{k+1}^\varepsilon\cdot\beta^{k-1}\cdot \int_{\Omega_0}w\nu^{-1}_\lambda d\nu_\lambda\\
&\leq\alpha_0^\epsilon\cdot\frac{1}{\nu_\lambda(B)}\int_B w\nu^{-1}_\lambda d\nu_\lambda\bigg(1+\sum_{k=1}^\infty(\frac{4^{\lambda+1}}{\alpha})^{(k+1)\varepsilon}\beta^{k-1}\bigg).
\end{align*}
Choose $\varepsilon>0$ small enough such that the infinite sum converges, then we have the desired result.
\end{proof}

\begin{cor}[Ratio Testing]~\\
Let $1<p<\infty$ and $w\in \widetilde{A}_{p,\lambda}$, then there is $C>0, \delta>0$ so that for all interval $B\subset\R_+$,
$$
A\subset B\implies \frac{w(A)}{w(B)}\leq C\left(\frac{\nu_\lambda(A)}{\nu_\lambda(B)}\right)^\delta.
$$
\end{cor}

\bigskip

\begin{proof}
Suppose $w$ satisfies the reverse type inequality with exponent $1+\varepsilon$. Then
 \begin{align*}
 \int_A w (\nu_\lambda)^{-1} d\nu_\lambda &\leq\nu_\lambda(A)^{\frac{\varepsilon}{1+\varepsilon}}\cdot\left(\int_A (w\nu^{-1}_\lambda)^{1+\varepsilon} d\nu_\lambda\right)^{\frac{1}{1+\varepsilon}}\\
&=\nu_\lambda(A)^{\frac{\varepsilon}{1+\varepsilon}} \nu_\lambda(B)^{\frac{1}{1+\varepsilon}}\cdot\left(\frac{1}{\nu_\lambda(B)}\int_B (w\nu^{-1}_\lambda)^{1+\varepsilon}d\nu_\lambda\right)^{\frac{1}{1+\varepsilon}}\\
&\leq C\nu_\lambda(A)^{\frac{\varepsilon}{1+\varepsilon}} \nu_\lambda(B)^{\frac{1}{1+\varepsilon}}\cdot\left(\frac{1}{\nu_\lambda(B)}\int_B w \right).
\end{align*}
This gives that if 
$ \delta := \frac{\varepsilon}{1+\varepsilon}$, then we have
$$ A\subset B\implies \frac{w(A)}{w(B)}\leq C\left(\frac{\nu_\lambda(A)}{\nu_\lambda(B)}\right)^\delta.
$$
The proof is complete.
\end{proof}

\bigskip

\begin{thm}[$L^p(\R_+,w)$-Boundedness for $\cm_\lambda$]~\\
Let $1<p<\infty$, and $f\in L^p(\R_+,w)$, then
$$
w\in\widetilde{A}_{p,\lambda}\Longleftrightarrow\|\cm_\lambda f\|_{L^p(\R_+,w)}\leq C\|f\|_{L^p(\R_+,w)}.$$
\end{thm}

\begin{proof}
By interpolation and the weak type estimate, it suffices to show that if $w \in \widetilde{A}_{p,\lambda}$, then $w \in \widetilde{A}_{q,\lambda}$ for some $q < p$.\\
Suppose $w\in\widetilde{A}_{p,\lambda}$, which is equivalent to $t^{(2\lambda+1)p^{'}}w(t)^{\frac{-1}{p-1}}\in\widetilde{A}_{p^{'},\lambda}$. Then by Reverse H\"older Inequality, we have for all $B$
$$
\left(\frac{1}{\nu_\lambda(B)}\int_B (t^{(2\lambda+1)p^{'}}w(t)^{\frac{-1}{p-1}}\nu^{-1}_\lambda)^{1+\varepsilon}d\nu_\lambda\right)^{\frac{p-1}{1+\varepsilon}}\leq C\left(\frac{1}{\nu_\lambda(B)}\int_B t^{(2\lambda+1)p^{'}}w(t)^{\frac{-1}{p-1}} dt \right)^{p-1}.
$$
Let $q-1:=\frac{p-1}{1+\varepsilon}$, then 
\begin{align*}
(t^{(2\lambda+1)p^{'}}w(t)^{\frac{-1}{p-1}}\nu^{-1}_\lambda)^{1+\varepsilon}
=w(t)^{\frac{-1}{q-1}} \cdot t^{(2\lambda+1)\frac{1}{q-1} },
\end{align*}
and hence
$$
\left(\frac{1}{\nu_\lambda(B)}\int_B t^{2(\lambda+1)q^{'}}w(t)^{\frac{-1}{q-1}}dt\right)^{q-1}\leq C\left(\frac{1}{\nu_\lambda(B)}\int_B t^{2(\lambda+1)p^{'}}w(t)^{\frac{-1}{p-1}} dt \right)^{p-1}.
$$
Thus,
$$
\left(\frac{1}{\nu_\lambda(B)}\int_B w(t) dt \right)\cdot \left(\frac{1}{\nu_\lambda(B)}\int_B t^{2(\lambda+1)q^{'}}w(t)^{\frac{-1}{q-1}}dt\right)^{q-1}\leq C,
$$
which implies that
$$
\left(\frac{1}{\nu_\lambda(B)}\int_B w(t) dt\right) \cdot \left(\frac{1}{\nu_\lambda(B)}\int_B t^{2(\lambda+1)q^{'}}w(t)^{\frac{-1}{q-1}}dt\right)^{q-1}\leq C.
$$
Thus, $w\in \widetilde{A}_{q,\lambda}$, for some $q<p.$
\end{proof}

\bigskip

\section{Characterization of the class $A_{p,\lambda}$ and the Reverse Type Inequality, $1<p<\infty$}

In this section, we will establish the reverse structure of $A_{p,\lambda}.$

\begin{thm}[Characterization of $A_{p,\lambda}$]~\\
For $1<p<\infty$, the weight $w$ is in $A_{p,\lambda}$ if and only if 
$$
    \mu(B)\left(\frac{1}{\nu_\lambda(B)}\int_B |f(t)| d\nu_\lambda (t)\right)^p\leq C\cdot \int_B |f|^p(t) d\mu(t),
$$
for all interval $B$ and $f\in L^p(\R_+,\mu)$.
\end{thm}
\bigskip

\begin{proof}
$(\Longrightarrow)$
For all interval $B\subset \R_+$,
\begin{align*}
\frac{1}{\nu_\lambda(B)}\int_I |f(t)| t^{2\lambda+1}dt&\leq \left(\frac{1}{\nu_\lambda(B)}\int_B|f(t)|^p d\mu(t)\right)^{1\over{p}}\left(\frac{1}{\nu_\lambda(B)}\int_B t^{2\lambda p'} w(t)^{\frac{-1}{p-1}}dt\right)^{1\over{p'}}\\
&\leq C \left(\frac{1}{\nu_\lambda(B)}\int_B|f(t)|^pd\mu(t)\right)^{1\over{p}}\left({\mu(B)\over\nu_\lambda(B)}\right)^{-1\over {p}},
\end{align*}
where $\mu(t)=t^p w(t)$.
That is
$$
    \mu(B)\left(\frac{1}{\nu_\lambda(B)}\int_B |f|d\nu_\lambda\right)^p\leq C\cdot \int_B |f|^p d\mu.
$$
$(\Longleftarrow)$
Conversely, we let
$$
f(t):=w(t)^{-1\over{p-1}}t^{{{1+2\lambda}\over{p-1}}-p^{'}},
$$
and hence
$$
\left(\frac{1}{\nu_\lambda(B)}\int_B t^{2\lambda p^{'}}w(t)^{-1\over{p-1}}dt\right)^{p-1} \left({\mu(B)\over{\nu_\lambda (B)}}\right)\leq C.
$$
The proof is complete.
\end{proof}

We also have the following structures for the weights $A_{p,\lambda}$. Recall again $d\mu(t)= t^pw(t)dt. $
\begin{lem}
Let $w\in A_{p,\lambda}$, where $1<p<\infty$. Then for all $\alpha\in (0,1)$, there is $\beta\in (0,1)$, so that for all interval $B$ with $S\subset B$, one has
$$
\nu_\lambda(S)\leq\alpha\nu_\lambda(B)\implies \mu(S)\leq\beta\mu(B).
$$
\end{lem}

 \medskip

\begin{thm}[Reverse Type Inequality for $A_{p,\lambda}$]~\\
Let $w\in A_{p,\lambda}$, $1<p<\infty.$ Then there is $\varepsilon>0$ such that for any interval $B\subset\mathbb{R}_+$,
$$\bigg(\frac{1}{\nu_\lambda(B)}\int_B (\mu\nu^{-1}_\lambda)^{1+\varepsilon}\,d\nu_\lambda\bigg )^{1/(1+\varepsilon)}\underset{p,w}{\lesssim}\frac{1}{\nu_\lambda(B)}\int_B \mu \nu^{-1}_\lambda d\nu_\lambda.$$
\end{thm}

\medskip

\begin{cor}
Let $1<p<\infty$ and $w\in A_{p,\lambda}$, then there is $C>0$, $\delta>0$ such that for any interval $B$,
$$
A\subset B\implies \frac{\mu(A)}{\mu(B)}\leq C\bigg(\frac{\nu_\lambda(A)}{\nu_\lambda(B)}\bigg)^\delta.
$$
\end{cor}

\medskip

The proof of the above properties follows from the first theorem in this section and the methods in the section $5$. However, these proofs seem to be not efficient enough to give us an explicit relation between $\varepsilon$ and $[w]_{A_{p,\lambda}}$. In order to do so, we need more subtle proofs which we now build up.

\medskip

\begin{thm}[Reverse Type Constant for $A_{p,\lambda}$ ]~\\
Let $w\in A_{p,\lambda}$, then there is a constant $c>0$ such that for any interval $B\subset\R_+$,
$$\left(\frac{1}{\nu_\lambda(B)}\int_B (\mu\nu^{-1}_\lambda)^{1+\varepsilon}\,d\nu_\lambda\right )^{1/(1+\varepsilon)}\leq\frac{2}{\nu_\lambda(B)}\int_B \mu \nu^{-1}_\lambda d\nu_\lambda,$$
where $$
\varepsilon:=1+\frac{1}{c\cdot [w]_{A_{p,\lambda}}}. 
$$
\end{thm}

\begin{proof}
 For any $\varepsilon>0$, let $w_1:=\mu\cdot\nu_\lambda$ and consider that
 \begin{align*}
 \frac{1}{\nu_\lambda(B)}\int_B (\mu\nu^{-1}_\lambda)^{1+\varepsilon}\,d\nu_\lambda
 &=\frac{\varepsilon}{\nu_\lambda(B)}\int_B \left(\int^{\mu\cdot\nu^{-1}_\lambda}_0 t^{\varepsilon-1}dt\right)\cdot \mu\nu^{-1}_\lambda d\nu_\lambda \\
 &=\frac{\varepsilon}{\nu_\lambda(B)}\int_{\R_+}t^\varepsilon \mu\left(\{x\in B: (\mu\cdot\nu^{-1}_\lambda)(x)>t\}\right)\frac{dt}{t}\\
 &= Term_1 +Term_2,
\end{align*}
in which 
$$
Term_1:=\frac{\varepsilon}{\nu_\lambda(B)}\int^{(w_1)_{B,\lambda}}_0 t^\varepsilon \mu\left(\{x\in B: (\mu\cdot\nu^{-1}_\lambda)(x)>t\}\right)\frac{dt}{t}.
$$
Note that
\begin{align*}
Term_1\leq\frac{\varepsilon}{\nu_\lambda(B)}\int^{(w_1)_{B,\lambda}}_0 t^\varepsilon \mu(B)\frac{dt}{t}=(w_1)^{1+\varepsilon}_{B,\lambda}.
\end{align*}

To estimate the second term,  
\begin{align*}
\left(\frac{1}{\nu_\lambda (B)}\int_B f(t)t^{1+2\lambda}dt\right)^p
&=\left(\frac{1}{\nu_\lambda (B)}\int_B f(t)w^{\frac{1}{p}}(t)\cdot t \cdot w^{\frac{-1}{p}}(t)\cdot t^{2\lambda}dt\right)^p\\ 
&\leq \nu^{-p}_\lambda (B)\cdot\int_B f^p(t) d\mu(t)\cdot\left(\int_B t^{2\lambda p^{'}} w^{\frac{-1}{p-1}} (t)dt\right)^{p-1},   
\end{align*}
which implies that
\begin{align*}
\left(\frac{1}{\nu_\lambda (B)}\int_B f(t)t^{1+2\lambda}dt\right)^p \cdot\mu(B)
&\leq \int_B f^p(t) d\mu(t)\cdot\frac{\mu(B)}{\nu_\lambda (B)}\cdot\left(\frac{1}{\nu_\lambda (B)}\int_B t^{2\lambda p^{'}} w^{\frac{-1}{p-1}} (t)dt\right)^{p-1}\\
&\leq\int_B f^p(t)d\mu(t)\cdot [w]_{A_{p,\lambda}}.   
\end{align*}
Then for any measurable subset $S\subset B$, one has
$$
\left(\frac{\nu_\lambda (S)}{\nu_\lambda (B)}\right)^p\leq \frac{\mu(S)}{\mu(B)}\cdot [w]_{A_{p,\lambda}}.
$$
Define the set 
\begin{align*}
S_B:=\left\{x\in B:w_1(x)\leq \frac{(w_1)_{B,\lambda}}{2[w]_{p,\lambda}}\right\}
\implies \left(\frac{\nu_\lambda (S_B)}{\nu_\lambda (B)}\right)^p &\leq \frac{[w]_{A_{p,\lambda}}}{\mu(B)}\cdot\int_{S_B}\frac{(w_1)_{B,\lambda}}{2[w]_{A_{p,\lambda}}}\nu_\lambda(t)dt \\
&=\frac{\nu_\lambda (S_B)}{2\nu_\lambda (B)},
\end{align*}
and hence
$$
2^{p-1}\cdot\left(\frac{\nu_\lambda (B-S_B)}{\nu_\lambda (B)}\right)^{p-1}\geq 
\left(\frac{\nu_\lambda (B)}{\nu_\lambda (B)}\right)^{p-1}-\left(\frac{\nu_\lambda (S_B)}{\nu_\lambda (B)}\right)^{p-1}\geq 1-\frac{1}{2}\geq \frac{1}{2},
$$
that is
$$
2^{\frac{p}{p-1}}\cdot\nu_\lambda\left(\left\{x\in B: w_1 (x)> \frac{(w_1)_{B,\lambda}}{2[w]_{A_{p,\lambda}}}\right\}\right)\geq \nu_\lambda (B).
$$
Besides, by the dyadic maximal function argument, there are disjoint dyadic cubes $B_k\subset B$ such that
$$t<(w_1)_{B_k,\lambda}=\frac{\mu(B_k)}{\nu_\lambda  (B_k)}\leq 4^{1+\lambda}t,$$
which implies that for all $t\geq (w_1)_{B,\lambda},$
\begin{align*}
 \mu\left(\{x\in B:w_1(x)>t\}\right)&\leq   \mu\left(\{x\in B:\mathcal{M}^d_\lambda w_1(x)>t\}\right)\\
 &\leq\sum_k \mu(B_k)\\
 &\leq4^{1+\lambda}\cdot t\sum_k \nu_\lambda(B_k)\\
 &\leq 2^{\frac{p}{p-1}}\cdot4^{1+\lambda}\cdot t\sum_k \nu_\lambda\left(\{x\in B_k:w_1(x)>\frac{(w_1)_{B,\lambda}}{2[w]_{A_{p,\lambda}}}\}\right)\\
 &\leq 2^{\frac{p}{p-1}}\cdot4^{1+\lambda}\cdot t \sum_k \nu_\lambda\left(\{x\in B_k:w_1(x)>\frac{t}{2[w]_{A_{p,\lambda}}}\}\right)\\
 &\leq 2^{\frac{p}{p-1}}\cdot4^{1+\lambda}\cdot t \cdot \nu_\lambda\left(\{x\in B:w_1(x)>\frac{t}{2[w]_{A_{p,\lambda}}}\}\right).\\
\end{align*}
Therefore,
\begin{align*}
 Term_2:&=\frac{\varepsilon}{\nu_\lambda(B)}\int^\infty_{(w_1)_{B,\lambda}} t^\varepsilon \mu\left(\{x\in B: (\mu\cdot\nu^{-1}_\lambda)(x)>t\}\right)\frac{dt}{t}\\
 &\leq \frac{ 2^{\frac{p}{p-1}}\cdot4^{1+\lambda}\cdot\varepsilon}{\nu_\lambda (B)}\int^\infty_{(w_1)_{B,\lambda}} t^{1+\varepsilon}\cdot\nu_\lambda\left(\{x\in B : w_1 (x)>\frac{t}{2[w]_{A_{p,\lambda}}}\}\right)\frac{dt}{t}\\
 &=\frac{2^{\frac{p}{p-1}}\cdot4^{1+\lambda}\cdot\varepsilon}{\nu_\lambda (B)}\cdot (2\cdot [w]_{A_{p,\lambda}})^{1+\varepsilon}\int^{\infty}_{\frac{(w_1)_{B,\lambda}}{2\cdot [w]_{A_{p,\lambda}}}}\alpha^{1+\varepsilon}\cdot\nu_\lambda\left(\{x\in B : w_1 (x)>\alpha\}\right)\frac{d\alpha}{\alpha}\\
 &\leq \frac{ 2^{\frac{p}{p-1}}\cdot4^{1+\lambda}\cdot\varepsilon}{\nu_\lambda (B)}\cdot (2\cdot [w]_{A_{p,\lambda}})^{1+\varepsilon} \int_B \nu_\lambda (x) \int^{w_1 (x)}_0 \alpha^\varepsilon d\alpha dx\\
 &= 2^{\frac{p}{p-1}}\cdot4^{1+\lambda}(2\cdot[w]_{A_{p,\lambda}})^{1+\varepsilon}\cdot \frac{\varepsilon}{1+\varepsilon}\cdot\frac{1}{\nu_\lambda (B)}\int_B w^{1+\varepsilon}_1(x)d\nu_\lambda (x).
\end{align*}
Choose $\varepsilon:=\frac{1}{c\cdot [w]_{A_p,\lambda}}$ and using the fact that $\forall t\geq 1\implies t^{\frac{1}{t}}\leq 2$, we have
\begin{align*}
 2^{\frac{p}{p-1}}\cdot4^{1+\lambda}(2\cdot[w]_{A_{p,\lambda}})^{1+\varepsilon}\cdot \frac{\varepsilon}{1+\varepsilon}&=\left(2\cdot[w]_{A_{p,\lambda}}\right)^{\frac{1}{2[w]_{A_{p,\lambda}}}\cdot\frac{2+2c [w]_{A_{p,\lambda}}}{c}} 2^{\frac{p}{p-1}}\cdot4^{1+\lambda}\cdot \frac{1}{1+c\cdot[w]_{A_{p,\lambda}}}\\
 &\leq 2^{\frac{p}{p-1}}\cdot\frac{4^{1+\lambda+\frac{1}{c}+[w]_{A_{p,\lambda}}}}{c} \\&< \frac{1}{2},
\end{align*}
provided that $c$ is large enough. As a result,
\begin{align*}
\frac{1}{\nu_\lambda(B)}\int_B (\mu\nu^{-1}_\lambda)^{1+\varepsilon}\,d\nu_\lambda&\leq (w_1)^{1+\varepsilon}_{B,\lambda}+Term_2\\
&\leq (w_1)^{1+\varepsilon}_{B,\lambda}+\frac{1}{2}\cdot \frac{1}{\nu_\lambda (B)}\int_B w^{1+\varepsilon}_1(x)d\nu_\lambda (x)
\end{align*}
$$
\implies \left(\frac{1}{\nu_\lambda(B)}\int_B (\mu\nu^{-1}_\lambda)^{1+\varepsilon}\,d\nu_\lambda\right)^{\frac{1}{1+\varepsilon}}\leq 2^{\frac{1}{1+\varepsilon}}\cdot\frac{1}{\nu_\lambda(B)}\int_B \mu \nu^{-1}_\lambda d\nu_\lambda\leq\frac{2}{\nu_\lambda(B)}\int_B \mu \nu^{-1}_\lambda d\nu_\lambda.
$$
The proof is complete.
\end{proof}

\section{Proof of Theorem \ref{thmcom} and Proposition \ref{propbmo}}
We first establish the following argument, where a prototype in $\R^n$ for the classical $A_p$ weight was studied in \cite{ho2011characterizations}.
\begin{prop}[$\rm BMO$ Norm-Equivalent]~\\
$(1)$ For all $b\in {\rm BMO}_\lambda (\R_+),$
$$
\frac{1}{2}\|b\|_{{\rm BMO}_\lambda (\R_+)}\leq \sup_{B\subset\R_+} \inf_{a\in \mathbb{C}}\frac{1}{\nu_\lambda(B)}\int_B |b(x)-a|d\nu_\lambda (x)\leq \|b\|_{{\rm BMO}_\lambda(\R_+)}.
$$
$(2)$ Let $\lambda>\frac{-1}{2}$ and $\lambda\neq 0$,
$    {\rm BMO}_{\Delta_\lambda} (\mathbb{R_+})$ coincides with ${\rm BMO}_\lambda(\R_+)$
and they have equivalent norms. 
\end{prop}

\begin{proof}
To prove $(1)$, let $a\in\mathbb{C}$ be given,
\begin{align*}
 \int_B |b(x)-b_{B,\lambda}|d\nu_\lambda\leq \int_B|b(x)-a|d\nu_\lambda+\int_B |a-b_{B,\lambda}|d\nu_\lambda  \leq2\cdot\int_B |b(x)-a|d\nu_\lambda,
\end{align*}
and the other inequality is trivial.

We now prove (2). Treat $(\R_+,|\cdot|,m_\lambda)$ as a space of homogeneous type in the sense of Coifman and Weiss \cite{coifman1977extensions}. Then by noting that
$$\nu_\lambda(t) = t^{2\lambda+1}dt =t\cdot t^{2\lambda}dt =W(t) m_\lambda(t),$$
the triple $(\R_+,|\cdot|,\nu_\lambda)$ is the weighted space of $(\R_+,|\cdot|,m_\lambda)$ under the standard weight $W(t)=t$.
One can verity that $W(t)$ is a Muckenhoupt $A_2$ weight in the setting 
$(\R_+,|\cdot|,m_\lambda)$. That is,
\begin{align*}
    \bigg(\frac{1}{m_\lambda(B)}\int_B W(t)dm_\lambda(t)\bigg)\bigg(\frac{1}{m_\lambda(B)}\int_B W(t)^{-1}dm_\lambda(t)\bigg)<C.
\end{align*}

We now first show that $\rm{ BMO}_{\Delta_\lambda}(\mathbb{R_+})\subset {\rm BMO}_\lambda(\R_+)$.  To see this, we first note that 
for ${\rm BMO}_{\Delta_\lambda} (\mathbb{R_+})$, we have the standard John--Nirenberg inequality as we treat $(\R_+,|\cdot|,m_\lambda)$ as a space of homogeneous type, that is,
$$ m_\lambda(\{ x\in B\subset \R_+: \left|b(x)-b_B\right|>\gamma\})
\leq C_1\exp\left(-{ C_2\gamma\over \|b\|_{{\rm BMO}_{\Delta_\lambda} (\mathbb{R_+})} }\right)\cdot m_\lambda(B).$$
Then we have
$$ W(\{ x\in B\subset \R_+: \left|b(x)-b_{B}\right|>\gamma\})
\leq C_0C_1^\varepsilon\exp\left({-{ C_2\cdot\varepsilon\gamma\over \|b\|_{{\rm BMO}_{\Delta_\lambda} (\mathbb{R_+})} }}\right)\cdot W(B).$$
Here, we recall again $W(B)$ is understood as the weight on $(\R_+,|\cdot|,m_\lambda)$, that is, 
$$W(B) = \int_B W(t) dm_\lambda(t) = \int_B t\cdot t^{2\lambda}dt
= \nu_\lambda(B).$$

Hence, for any $B\subset\R_+$,
\begin{align*}
\frac{1}{\nu_\lambda(B)}\int_B |b(t)-b_{B,\lambda}|d\nu_\lambda (t) &\leq \frac{2}{\nu_\lambda(B)}\int_B |b(t)-b_{B}|d\nu_\lambda (t)\\
&=
    {2\over W(B)}\int_B |b(t)-b_{B}|W(t)dm_\lambda(t)\\
    &={2\over W(B)} \int_0^\infty W(\{ x\in B: |b(x)-b_{B}|>\gamma\})d\gamma\\
    &\leq  2C_0C_1^\varepsilon \int_0^\infty \exp\left({-{ C_2\cdot\varepsilon\gamma\over \|b\|_{{\rm BMO}_{\Delta_\lambda} (\mathbb{R_+})} }}\right) d\gamma\\
    &\lesssim  \|b\|_{{\rm BMO}_{\Delta_\lambda} (\mathbb{R_+})}.
\end{align*}
Thus, we have
$$\|b\|_{{\rm BMO}_{\Delta_\lambda} (\mathbb{R_+})}\subset {\rm BMO}_\lambda(\R_+).$$

Next, we prove ${\rm BMO}_\lambda(\R_+)\subset {\rm BMO}_{\Delta_\lambda} (\mathbb{R_+})$.

 By H\"older's inequality, we have
 \begin{align*}
     &{1\over m_\lambda(B)}\int_B |b(t)-b_B|dm_\lambda(t)\\
     &\leq{2\over m_\lambda(B)}\int_B |b(t)-b_{B,\lambda}|dm_\lambda(t)\\
     &\leq{1\over m_\lambda(B)}\left(\int_B |b(t)-b_{B,\lambda}|^2W(t)dm_\lambda(t)\right)^{1\over2}\left(\int_B W(t)^{-1}dm_\lambda(t)\right)^{1\over2}\\
     &\lesssim \left({1\over W(B)}\int_B |b(t)-b_{B,\lambda}|^2W(t)dm_\lambda(t)\right)^{1\over2},
 \end{align*}
 where the last inequality  follows from the $A_2$ condition. Finally using John--Nirenberg inequality, we have $$\left({1\over W(B)}\int_B |b(t)-b_{B,\lambda}|^2W(t)dm_\lambda(t)\right)^{1\over2} \lesssim \|b\|_{{\rm BMO}_{\lambda} (\mathbb{R_+})} .$$

 The proof is complete.
\end{proof}

\medskip

With the development of the new average and the new weighted class, $\widetilde{A}_{p,\lambda}$ and $A_{p,\lambda}$, we can characterize the $L^p(\R_+,w)$-boundedness of the commutator $[b, R_\lambda]$. Our method in showing the upper bound is ignited by the powerful Cauchy Integral trick and the inner structure of $A_{p,\lambda}$. On the other hand, we apply the method in \cite{duong2021two} to prove the lower bound of $[b, R_\lambda].$ First off, we show that

$$
w\in A_{p,\lambda} \implies
\|[b,R_\lambda]\|_{L^p(\mathbb{R_+},w)\to L^p(\mathbb{R_+},w)} \lesssim \|b\|_{{\rm BMO}_\lambda} <\infty.
$$

\bigskip

\begin{proof}
Pick $z\in\mathbb{C}$ and define the generating operator by the following
$$
T_z (f):=e^{zb}\cdot R_\lambda (e^{-zb}f).
$$
Note that for appropriate function $f$, we have
\begin{align*}
T_z (f)&=\left(1+zb+\cdots\right)\cdot R_\lambda \left((1-zb+\cdots)f\right)\\
&= R_\lambda (f)+z\cdot \left(b R_\lambda (f)-R_\lambda(bf)\right)+z^2\cdot G\\
&= R_\lambda (f)+z\cdot [b,R_\lambda]f+z^2\cdot G,
\end{align*}
where $G$ is an operator, and hence
$$
[b,R_\lambda]f=\frac{d}{dz}\mid_{z=0} T_z (f).
$$
By Cauchy Integral Formula, 
$$
[b,R_\lambda]f=\frac{1}{2\pi i}\int_{D_\varepsilon}\frac{T_z (f)}{z^2}dz,
$$
in which $D_\varepsilon$ is the boundary of a disk with radius $ \varepsilon>0.$
Then by Minkowski's integral inequality,
\begin{align}\label{commutator upper bound}
\|[b,R_\lambda]f\|_{L^p(R_+,w)}\leq \frac{1}{2\pi \varepsilon^2}\int_{D_\varepsilon}\|T_z(f)\|_{L^p(R_+,w)}dz.
\end{align}
Now, we calculate the term inside in integral and get
$$
\|T_z(f)\|_{L^p(R_+,w)}=\|R_\lambda (e^{-zb}f)\|_{L^p(R_+,e^{pRe(z)b}w)}.
$$
Since $\|R_\lambda (f)\|_{L^p(R_+,w)}\underset{A_{p,\lambda}}{\lesssim}\|f\|_{L^p(R_+,w)}$, we need to check that whether
\begin{align}\label{exp zb weight}
[e^{pRe(z)b}w]_{A_{p,\lambda}}<\infty.    
\end{align}

For any interval $B\subset\R_+$, consider the quantity
\begin{align}\label{Ap quantity}
\bigg(\frac{1}{\nu_\lambda (B)}\int_Bt^{p}w(t)e^{pRe(z)b}dt\bigg)\cdot\bigg(\frac{1}{\nu_\lambda (B)}\int_B t^{2\lambda p^{'} }w(t)^{1-p^{'}}e^{-p^{'}Re(z)b}dt\bigg)^{p-1}.
\end{align}

\bigskip

For the first term,
\begin{align*}
 &\frac{1}{\nu_\lambda (B)}\int_Bt^{p}w(t)e^{pRe(z)b}dt\\
 &=\frac{1}{\nu_\lambda (B)}\int_B\frac{t^p w(t)}{\nu_\lambda(t)}e^{p Re(z)b}d\nu_\lambda(t)\\
 &\leq \left(\frac{1}{\nu_\lambda (B)}\int_B\left(\frac{t^p w(t)}{\nu_\lambda(t)}\right)^{1+\varepsilon_1}d\nu_\lambda(t)\right)^{\frac{1}{1+\varepsilon_1}}\cdot\left(\frac{1}{\nu_\lambda (B)}\int_Be^{(1+\frac{1}{\varepsilon_1})pRe(z)b}d\nu_\lambda(t)\right)^{\frac{\varepsilon_1}{1+\varepsilon_1}}\\
 &\lesssim \left(\frac{1}{\nu_\lambda (B)}\int_B t^p w(t)dt\right)\cdot\left(\frac{1}{\nu_\lambda (B)}\int_Be^{(1+\frac{1}{\varepsilon_1})pRe(z)b}d\nu_\lambda(t)\right)^{\frac{\varepsilon_1}{1+\varepsilon_1}},
\end{align*}
where the second inequality holds by the reverse type inequality.

\bigskip

For the second term, we use the dual factor to argue. Note that 
$$w\in A_{p,\lambda} \Longleftrightarrow t^{(2\lambda-1)p'}w^{-1\over p-1}\in A_{p',\lambda}.
$$
So,
\begin{align*}
 \frac{1}{\nu_\lambda (B)}\int_Bt^{2\lambda p^{'}}w^{1-p^{'}}(t)e^{-p^{'}Re(z)b}dt
 =\frac{1}{\nu_\lambda (B)}\int_B
 \frac{t^{2\lambda p^{'}} w^{1-p^{'}}(t)}{\nu_\lambda(t)}e^{-p^{'} Re(z)b}d\nu_\lambda(t),
 \end{align*}
 applying H\"older's inequality and the reverse-type inequality, we have
 \begin{align*}
 &\leq \left(\frac{1}{\nu_\lambda (B)}\int_B\left(\frac{t^{2\lambda p^{'} } w^{1-p^{'} }(t)}{\nu_\lambda(t)}\right)^{1+\delta}d\nu_\lambda(t)\right)^{\frac{1}{1+\delta}}\left(\frac{1}{\nu_\lambda (B)}\int_Be^{-(1+\frac{1}{\delta})p^{'}Re(z)b}d\nu_\lambda(t)\right)^{\frac{\delta}{1+\delta}}\\
 &\lesssim \left(\frac{1}{\nu_\lambda (B)}\int_B t^{2\lambda p^{'}} w^{1-p^{'} }(t)dt\right)\cdot\left(\frac{1}{\nu_\lambda (B)}\int_Be^{-(1+\frac{1}{\delta})p^{'} Re(z)b}d\nu_\lambda(t)\right)^{\frac{\delta}{1+\delta}}.
\end{align*}

\medskip
Note that the constant $\varepsilon_1$ and $\delta$ can be chosen to be equal.
In conclusion, the quantity \eqref{Ap quantity} is bounded by 
\begin{align}\label{weight estimate 8.2}
&\left(\frac{1}{\nu_\lambda (B)}\int_B t^p w(t)dt\right)
\left(\frac{1}{\nu_\lambda (B)}\int_B t^{2\lambda p^{'}} w^{1-p^{'} }(t)dt\right)^{p-1}\\
&\qquad\times \left(\frac{1}{\nu_\lambda (B)}\int_Be^{(1+\frac{1}{\varepsilon_1})p Re(z)b}d\nu_\lambda(t)\right)^{\frac{\varepsilon_1}{1+\varepsilon_1}}\cdot\left(\frac{1}{\nu_\lambda (B)}\int_Be^{-(1+\frac{1}{\varepsilon_1})p^{'} Re(z)b}d\nu_\lambda(t)\right)^{\frac{\varepsilon_1}{1+\varepsilon_1} (p-1)}\nonumber\\
&\leq [w]_{A_{p,\lambda}}\cdot\left(\frac{1}{\nu_\lambda (B)}\int_Be^{(1+\frac{1}{\varepsilon_1})p Re(z)b}d\nu_\lambda(t)\right)^{\frac{\varepsilon_1}{1+\varepsilon_1}}\cdot\left(\frac{1}{\nu_\lambda (B)}\int_Be^{-(1+\frac{1}{\varepsilon_1})p^{'} Re(z)b}d\nu_\lambda(t)\right)^{\frac{\varepsilon_1}{1+\varepsilon_1} (p-1)}.\nonumber
\end{align}


Hence, if we choose $\varepsilon_1>0$ such that
$$
\bigg|(1+\frac{1}{\varepsilon_1})\cdot p\cdot Re(z)\bigg|\leq\frac{A}{\|b\|_{{\rm BMO}_\lambda}}\cdot\min\{1,p-1\},
$$
then together with (2) of Corollary \ref{cor exp}, the estimate \eqref{weight estimate 8.2} implies that 
$$
[e^{p Re(z)b}w]_{A_{p,\lambda}}\lesssim[w]_{A_{p,\lambda}}\cdot \big(C(A,\varepsilon,\varepsilon_1)\big)^{p},
$$
where $C(A,\varepsilon,\varepsilon_1)>1.$ Thus, we see that \eqref{exp zb weight} holds.

\bigskip
Now we return to the estimate \eqref{commutator upper bound}.
If we let $$\varepsilon=\frac{A}{p(1+\frac{1}{\varepsilon_1})\|b\|_{{\rm BMO}_\lambda}}\cdot\min\{1,p-1\} \implies    \left|(1+\frac{1}{\varepsilon_1})p\, Re(z)\right|\leq\left|(1+\frac{1}{\varepsilon_1})p\varepsilon\right|=\frac{A}{\|b\|_{{\rm BMO}_\lambda}}\cdot\min\{1,p-1\},
$$
which implies that
\begin{align*}
 \|[b,R_\lambda]f\|_{L^p(R_+,w)}&\leq \frac{1}{2\pi \varepsilon^2}\int_{D_\varepsilon}\|T_z(f)\|_{L^p(R_+,w)}dz\\
&\lesssim  \frac{1}{\varepsilon}[w]_{A_{p,\lambda}}\|f\|_{L^p(R_+,w)}.
\end{align*}
That is 
$$
\|[b,R_\lambda]\|_{L^p(R_+,w)\to L^p(R_+,w)} \lesssim [w]_{A_{p,\lambda}}\|b\|_{{\rm BMO}_\lambda}.
$$
The proof is complete.
\end{proof}

\bigskip

\bigskip

In the last part of our paper, we prove the weighted lower bound of the commutator. In Section 7.5 of \cite{duong2021two}, the authors proved the following argument holds for the kernel $K(x,y)$ of the operator $R_\lambda$:

There exist positive constants $3\le A_1\le A_2$ such that for any interval $B:=B(x_0, r)\subset \R_+$, there exists another interval $\widetilde B:=B(y_0, r)\subset \R_+$
such that $A_1 r\le |x_0- y_0|\le A_2 r$, 
and  for all $(x,y)\in ( B\times \widetilde{B})$, $K(x, y)$ does not change sign and
\begin{equation}\label{e-assump cz ker low bdd}
|K(x, y)|\gs \frac1{m_\lambda(\tilde B)},
\end{equation}
where $dm_\lambda:=x^{2\lambda}dx$ with $dx$ the Lebesgue measure on $\R_+$.

\begin{defn}\label{d-median value}
  By a median value of a real-valued measurable function $f$ over an interval $B\subset\R_+$, we mean a possibly non-unique, real number $\alpha_B(f)$ such that
$$m_\lambda(\{x\in B: f(x)>\alpha_B(f)\})\leq \frac12m_\lambda(B)\,\, \, \mbox{and}\,\,\, m_\lambda(\{x\in B: f(x)<\alpha_B(f)\})\leq \frac12m_\lambda(B). $$
\end{defn}
It is known that for a given function $f$ and interval $B$, the median value exists and may not be unique; see, for example, \cite{journe1983}. The following decomposition follows from the result in Section 5 of \cite{duong2021two}.

\begin{lem}\label{l-bmo decomp low bdd}
  Let $b$ be a real-valued measurable function.
For any ball $B$, let $\widetilde B$ be as in \eqref{e-assump cz ker low bdd}. Then
there exist measurable sets $E_1, E_2\subset  B$, and $F_1, F_2\subset \widetilde B$, such that
\begin{itemize}
  \item[{\rm (i)}] $ B=E_1\cup E_2$,  $\widetilde B=F_1\cup F_2$ and $m_\lambda(F_i)\ge\frac{1}{2}m_\lambda(\widetilde B)$, $i=1,2$;
  \item[{\rm (ii)}]  $b(x)-b(y)$ does not change sign for all $(x, y)$ in $E_i\times F_i$, $i=1,2$;
  \item[{\rm (iii)}]  $|b(x)-\alpha_{\widetilde B}(b)|\leq |b(x)-b(y)|$ for all $(x, y)$ in $E_i\times F_i$, $i=1,2$.
\end{itemize}
\end{lem}

\medskip

We now prove the lower bound for $[b,R_\lambda]$, that is 
$$
w\in {A_{p,\lambda}}\implies \|b\|_{{\rm BMO}_\lambda}\lesssim\|[b,R_\lambda]\|_{L^p(\R_+,w)\to L^p(\R_+,w)}
$$

\begin{proof}
For given $b\in L^1_{\rm loc}(\R_+,\nu_\lambda)$ and for any ball $B$, 
we will show that for any ball $B$,
\begin{align}\label{e-mean osci weigh upp bdd}
 \sup_{B\subset \mathbb R_+} {1\over \nu_\lambda(B)}\int_B|b(t)-b_{B,\lambda}|d\nu_\lambda(t) \ls \|[b,R_\lambda]\|_{L^p(\R_+,w)\to L^p(\R_+,w)}.
\end{align}

Without loss of generality, 
let $B$ be an interval in $\R_+$. We apply the conditions \eqref{e-assump cz ker low bdd} and Lemma \ref{l-bmo decomp low bdd}
to get sets $E_i, F_i,\,i=1,2$.

On the one hand, by Lemma \ref{l-bmo decomp low bdd} and \eqref{e-assump cz ker low bdd}, we have that for $f_i:=\chi_{F_i}$, $i=1,2$,
\begin{align*}
\frac1{\nu_\lambda(B)}\sum_{i=1}^2\int_{ B} |[b,R_\lambda]f_i(x) |\,d\nu_\lambda(x)
&\ge\frac1{\nu_\lambda(B)}\sum_{i=1}^2\int_{E_i} |[b,R_\lambda]f_i(x) |\,d\nu_\lambda(x)\\
&=\frac1{\nu_\lambda(B)}\sum_{i=1}^2\int_{E_i}\int_{F_i}|b(x)-b(y)||K(x,y)|\,dm_\lambda(y)\,d\nu_\lambda(x)\\
&\gs\frac1{\nu_\lambda(B)}\sum_{i=1}^2\int_{E_i}\int_{F_i}\frac{|b(x)-\alpha_{\widetilde B}(b)|}{m_\lambda(\tilde B)}\,dm_\lambda(y)\,d\nu_\lambda(x)\\
&\sim\frac1{\nu_\lambda(B)}\sum_{i=1}^2\int_{E_i}\frac{|b(x)-\alpha_{\widetilde B}(b)|}{m_\lambda(\tilde B)}\,d\nu_\lambda(x)\cdot m_\lambda(F_i)\\
&\gs \frac1{\nu_\lambda(B)}\int_B\lf|b(x)-\alpha_{\widetilde B}(b)\r|\,d\nu_\lambda(x)\\
&\gs  {1\over \nu_\lambda(B)}\int_B|b(x)-b_{B,\lambda}|d\nu_\lambda(x).
\end{align*}

On the other hand, from H\"older's inequality and the boundedness of $[b,R_\lambda]$, we deduce that
\begin{align*}
&\frac1{\nu_\lambda(B)}\sum_{i=1}^2\int_{ B} |[b,R_\lambda]f_i(x)|\,d\nu_\lambda(x)\\
&=\frac1{\nu_\lambda(B)}\sum_{i=1}^2\int_{ B} |[b,R_\lambda]f_i(x)| w(x)^{1\over p} \, w(x)^{-{1\over p}}\ x^{2\lambda+1}dx\\
&\leq \frac{1}{\nu_\lambda (B)}\sum^2_{i=1} \left(\int_B |[b,R_\lambda]f_i(x)|^p w(x)dx\right)^{\frac{1}{p}}\cdot\left(\int_B x^{(2\lambda+1)p^{'}}\cdot w^{\frac{-1}{p-1}}(x)dx\right)^{\frac{1}{p^{'}}}\\
&\lesssim\frac{1}{\nu_\lambda (B)}\sum^2_{i=1}\|[b,R_\lambda]\|_{L^p(\R_+,w)\to L^p(\R_+,w)}\cdot w(F_i)^{\frac{1}{p}}\cdot\left(\int_Bx^{(2\lambda+1)p^{'}}w^{\frac{-1}{p-1}}(x)dx\right)^{\frac{1}{p^{'}}}\\
&\lesssim \frac{1}{\nu_\lambda (B)}\|[b,R_\lambda]\|_{L^p(\R_+,w)\to L^p(\R_+,w)}\cdot w(\widetilde B)^{\frac{1}{p}}\cdot\left(\int_Bx^{(2\lambda+1)p^{'}}w^{\frac{-1}{p-1}}(x)dx\right)^{\frac{1}{p^{'}}}.
\end{align*}

 Denote $B=(a,b)$ and $ \widetilde{B}=(a',b')$ and notice that from the construction of $\widetilde{B}$ in \cite{duong2021two}, $\widetilde{B}$ can always be chosen on the right hand side of $B$. Hence, let $B^*=(a,b')$, we have
\begin{align*}
&\frac{1}{\nu_\lambda (B)}\|[b,R_\lambda]\|_{L^p(\R_+,w)\to L^p(\R_+,w)}\cdot w(\widetilde B)^{\frac{1}{p}}\cdot\left(\int_Bx^{(2\lambda+1)p^{'}}w^{\frac{-1}{p-1}}(x)dx\right)^{\frac{1}{p^{'}}}\\
& \leq \frac{1}{\nu_\lambda (B)}\|[b,R_\lambda]\|_{L^p(\R_+,w)\to L^p(\R_+,w)}\cdot\left(\int_{\widetilde{B}} b^p w(t)dt\right)^{\frac{1}{p}}\cdot\left(\int_B x^{2\lambda p^{'}}w^{\frac{-1}{p-1}}(x)dx\right)^{\frac{1}{p^{'}}}\\
&= \frac{1}{\nu_\lambda (B)}\|[b,R_\lambda]\|_{L^p(\R_+,w)\to L^p(\R_+,w)}\cdot\left(\int_{\widetilde{B}} \frac{b^p}{t^p}\cdot t^p w(t)dt\right)^{\frac{1}{p}}\cdot\left(\int_B x^{2\lambda p^{'}}w^{\frac{-1}{p-1}}(x)dx\right)^{\frac{1}{p^{'}}}\\
&\leq \frac{1}{\nu_\lambda (B)}\|[b,R_\lambda]\|_{L^p(\R_+,w)\to L^p(\R_+,w)}\cdot\left(\int_{B^*} t^p w(t)dt\right)^{\frac{1}{p}}\cdot\left(\int_{B^*} x^{2\lambda p^{'}}w^{\frac{-1}{p-1}}(x)dx\right)^{\frac{1}{p^{'}}}\\
&\leq \frac{\nu_\lambda (B^*)}{\nu_\lambda (B)}\|[b,R_\lambda]\|_{L^p(\R_+,w)\to L^p(\R_+,w)}\cdot [w]_{A_{p,\lambda}}.
\end{align*}
Since $\nu_\lambda$ is doubling and the sizes of $B$ and $B^*$ are comparable and their distance is also comparable concerning their radius, we then have
$$
\|b\|_{{\rm BMO}_\lambda}\lesssim [w]_{A_{p,\lambda}} \|[b,R_\lambda]\|_{L^p(\R_+,w)\to L^p(\R_+,w)}.
$$
Thus, \eqref{e-mean osci weigh upp bdd} holds and the proof of the Theorem is complete.
\end{proof}

\bigskip

\begin{cor}
Suppose $b\in L^1_{loc}(\R_+)$ and let $1<p<\infty$. Then if $b\in {\rm BMO}_{\lambda}(\R_+)$,
$$
w\in {A_{p,\lambda}}\implies \|[b,R_\lambda]\|_{L^p(\R_+,w)\to L^p(\R_+,w)}\sim \|b\|_{{\rm BMO}_\lambda}.
$$
\end{cor}

\bigskip
\bigskip
\bigskip
\bigskip

{\bf Acknowledgments:} J. Li is supported by ARC DP 220100285. C.-W. Liang and C.-Y. Shen are supported in part by NSTC through grant 111-2115-M-002-010-MY5.  Fred Yu-Hsiang Lin is supported by the
DFG under Germany’s Excellence Strategy - EXC-2047/1 - 390685813 and
DAAD Graduate School Scholarship Programme - 57572629.

\bigskip

\printbibliography[heading=bibintoc,title={References}]

@article{muckenhoupt1965classical,
  title={Classical expansions and their relation to conjugate harmonic functions},
  author={Muckenhoupt, B. and Stein, E. M.},
  journal={Transactions of the American Mathematical Society},
  volume={118},
  pages={17--92},
  year={1965}
}

@article{coifman1976factorization,
  title={Factorization theorems for Hardy spaces in several variables},
  author={Coifman, R. R. and Rochberg, R. and Weiss, G.},
  journal={Annals of Mathematics},
  volume={103},
  number={3},
  pages={611--635},
  year={1976},
  publisher={JSTOR}
}

@article{calderon1976inequalities,
  title={Inequalities for the maximal function relative to a metric},
  author={Calder{\'o}n, A. P.},
  journal={Studia Mathematica},
  volume={57},
  number={3},
  pages={297--306},
  year={1976}
}

@article{coifman1977extensions,
  title={Extensions of Hardy spaces and their use in analysis},
  author={Coifman, R. R. and Weiss, G.},
  journal={Bulletin of the American Mathematical Society},
  volume={83},
  number={4},
  pages={569--645},
  year={1977}
}

@article{kerman1978boundedness,
  title={Boundedness criteria for generalized Hankel conjugate transformations},
  author={Kerman, R. A.},
  journal={Canadian Journal of Mathematics},
  volume={30},
  number={1},
  pages={147--153},
  year={1978},
  publisher={Cambridge University Press}
}

@article{andersen1981weighted,
  title={Weighted norm inequalities for generalized Hankel conjugate transformations},
  author={Andersen, K. and Kerman, R.},
  journal={Studia Mathematica},
  volume={71},
  pages={15--26},
  year={1981},
  publisher={Instytut Matematyczny Polskiej Akademii Nauk}
}

@article{journe1983,
  title={Calderón-Zygmund Operators, Pseudodifferential Operators and the Gauchy Integral of Calderón},
  author={J. L. Journé},
  journal={ Springer–Verlag},
  volume={},
  year={1983}
}

@article{betancor2007riesz,
  title={Riesz transforms related to Bessel operators},
  author={Betancor, J. J. and Fari{\~n}a, J. C. and Buraczewski, D. and Mart{\'\i}nez, T. and Torrea, J. L.},
  journal={Proceedings of the Royal Society of Edinburgh Section A: Mathematics},
  volume={137},
  number={4},
  pages={701--725},
  year={2007},
  publisher={Royal Society of Edinburgh Scotland Foundation}
}

@article{villani2008riesz,
  title={Riesz transforms associated to Bessel operators},
  author={Villani, M.},
  journal={Illinois Journal of Mathematics},
  volume={52},
  number={1},
  pages={77--89},
  year={2008},
  publisher={Duke University Press}
}

@article{betancor2009littlewood,
  title={On Littlewood-Paley functions associated with Bessel operators},
  author={Betancor, J. J. and Fari{\~n}a, J. C. and Sanabria, A.},
  journal={Glasgow Mathematical Journal},
  volume={51},
  number={1},
  pages={55--70},
  year={2009},
  publisher={Cambridge University Press}
}

@article{betancor2010maximal,
  title={Maximal operators, Riesz transforms and Littlewood--Paley functions associated with Bessel operators on BMO},
  author={Betancor, J. J. and Ruiz, A. C. and Fari{\~n}a, J. C. and Rodr{\'\i}guez-Mesa, L.},
  journal={Journal of mathematical analysis and applications},
  volume={363},
  number={1},
  pages={310--326},
  year={2010},
  publisher={Elsevier}
}

@article{betancor2010mapping,
  title={Mapping properties of fundamental operators in harmonic analysis related to Bessel operators},
  author={Betancor, J. J. and Harboure, E. and Nowak, A. and Viviani, B.},
  journal={Studia Math},
  volume={197},
  pages={101--140},
  year={2010}
}

@article{yang2011real,
  title={Real-variable characterizations of Hardy spaces associated with Bessel operators},
  author={Yang, D. and Yang, D.},
  journal={Analysis and Applications},
  volume={9},
  number={03},
  pages={345--368},
  year={2011},
  publisher={World Scientific}
}

@article{ho2011characterizations,
  title={Characterizations of BMO by $ A_p$ Weights and $p$-Convexity},
  author={Ho, K. P. },
  journal={Hiroshima Mathematical Journal},
  volume={41},
  number={2},
  pages={153--165},
  year={2011},
  publisher={Hiroshima University, Mathematics Program}
}

@article{betancor2015umd,
  title={UMD-valued square functions associated with Bessel operators in Hardy and BMO spaces},
  author={Betancor, J. J. and Castro, A. J. and Rodr{\'\i}guez-Mesa, L.},
  journal={Integral equations and operator theory},
  volume={81},
  number={3},
  pages={319--374},
  year={2015},
  publisher={Springer}
}

@article{duong2017factorization,
  title={Factorization for Hardy spaces and characterization for BMO spaces via commutators in the Bessel setting},
  author={Duong, X. T. and Li, J. and Wick, B. D. and Yang, D.},
  journal={Indiana University Mathematics Journal},
  pages={1081--1106},
  year={2017},
  publisher={JSTOR}
}

@article{duong2018compactness,
  title={Compactness of Riesz transform commutator associated with Bessel operators},
  author={Duong, X. T. and Li, J. and Mao, S. and Wu, H. and Yang, D.},
  journal={Journal d'Analyse Math{\'e}matique},
  volume={135},
  pages={639--673},
  year={2018},
  publisher={Springer}
}

@article{betancor2020bellman,
  title={Bellman functions and dimension free Lp-estimates for the Riesz transforms in Bessel settings},
  author={Betancor, J. J. and Dalmasso, E. and Fari{\~n}a, J. C. and Scotto, R.},
  journal={Nonlinear Analysis},
  volume={197},
  pages={111850},
  year={2020},
  publisher={Elsevier}
}

@article{duong2021two,
  title={Two weight commutators on spaces of homogeneous type and applications},
  author={Duong, X. T. and Gong, R. and Kuffner, M. S. and Li, J. and Wick, B. D. and Yang, D.},
  journal={The Journal of Geometric Analysis},
  volume={31},
  pages={980--1038},
  year={2021},
  publisher={Springer}
}

@article{chen2022square,
  title={Square roots of the Bessel operators and the related Littlewood--Paley estimates},
  author={Chen, Y. and Duong, X. T. and Li, J. and Tao, W. and Yang, D.},
  journal={Studia Mathematica},
  volume={263},
  pages={19--58},
  year={2022},
  publisher={Instytut Matematyczny Polskiej Akademii Nauk}
}

(J Li) Department of Mathematics, Macquarie University, NSW, 2109, Australia.\\ 
{\it E-mail}: \texttt{ji.li@mq.edu.au}\\

(C-W Liang) Department of Mathematics, National Taiwan University, Taiwan.
\\ 
{\it E-mail}: \texttt{d10221001@ntu.edu.tw}\\

(F Y-H Lin) Department of Mathematics, University of Bonn, Germany.
 \\ {\it E-mail}: \texttt{fredlin@math.uni-bonn.de}\\
 
(C-Y Shen) Department of Mathematics, National Taiwan University, Taiwan. \\ {\it E-mail}: \texttt{cyshen@math.ntu.edu.tw}

\vspace{0.3cm}

\end{document}